\documentclass[letterpaper]{article}

\usepackage{charter}
\usepackage[utf8]{inputenc}
\usepackage[T1]{fontenc}
\usepackage[dvipsnames]{xcolor}
\usepackage{hyperref}
\usepackage{enumitem}
\usepackage[super]{nth}
\usepackage[gen]{eurosym}
\usepackage{graphicx}
\usepackage{amssymb}
\usepackage{amsmath}
\usepackage{stmaryrd}
\usepackage{algorithm}
\usepackage{algorithmic}
\usepackage{dsfont}
\usepackage{mathtools}
\usepackage{tikz}
\usepackage[round]{natbib}
\usepackage{apxproof}
\usepackage{booktabs}
\usepackage{amsfonts}
\usepackage{amsthm}
\usepackage[font=small]{caption}

\renewcommand{\leq}{\leqslant}
\renewcommand{\geq}{\geqslant}
\renewcommand{\epsilon}{\varepsilon}

\DeclareMathOperator*{\argmin}{\mathrm{arg\,min}}
\DeclareMathOperator*{\argmax}{\mathrm{arg\,max}}
\DeclarePairedDelimiter\floor{\lfloor}{\rfloor}
\DeclarePairedDelimiter{\ceil}{\lceil}{\rceil}

\newtheorem{theorem}{Theorem}[section]
\newtheorem{corollary}[theorem]{Corollary}
\newtheorem{fact}[theorem]{Fact}
\newtheorem{lemma}[theorem]{Lemma}
\newtheorem{proposition}[theorem]{Proposition}
\newtheorem{remark}[theorem]{Remark}

\hypersetup{
    colorlinks,
    linkcolor={blue!50!black},
    citecolor={blue!50!black}
}

\begin{document}

\begin{center}
 {\Large Boosting Frank-Wolfe by Chasing Gradients}
\end{center}

\vspace{7mm}

\noindent\textbf{Cyrille W.~Combettes}\hfill\href{mailto:cyrille@gatech.edu}{\ttfamily cyrille@gatech.edu}\\
\emph{\small School of Industrial and Systems Engineering\\
Georgia Institute of Technology\\
Atlanta, GA, USA}\\
\\
\textbf{Sebastian Pokutta}\hfill\href{mailto:pokutta@zib.de}{\ttfamily pokutta@zib.de}\\
\emph{\small Institute of Mathematics and Department for AI in Society, Science, and Technology\\
Technische Universit\"at Berlin and Zuse Institute Berlin\\
Berlin, Germany}

\vspace{5mm}

\begin{center}
\begin{minipage}{0.85\textwidth}
\begin{center}
 \textbf{Abstract}
\end{center}
 {\small The Frank-Wolfe algorithm has become a popular first-order optimization algorithm for it is simple and projection-free, and it has been successfully applied to a variety of real-world problems. Its main drawback however lies in its convergence rate, which can be excessively slow due to naive descent directions. We propose to speed up the Frank-Wolfe algorithm by better aligning the descent direction with that of the negative gradient via a subroutine. This subroutine chases the negative gradient direction in a matching pursuit-style while still preserving the projection-free property. Although the approach is reasonably natural, it produces very significant results. We derive convergence rates $\mathcal{O}(1/t)$ to $\mathcal{O}(e^{-\omega t})$ of our method and we demonstrate its competitive advantage both per iteration and in CPU time over the state-of-the-art in a series of computational experiments.}
\end{minipage}
\end{center}

\vspace{0mm}

\section{Introduction}

Let $(\mathcal{H},\langle\cdot,\cdot\rangle)$ be a Euclidean space. In this paper, we address the constrained convex optimization problem
\begin{align}
 \min_{x\in\mathcal{C}}f(x)\label{pb}
\end{align}
where $f:\mathcal{H}\rightarrow\mathbb{R}$ is a smooth convex function and $\mathcal{C}\subset\mathcal{H}$ is a compact convex set. A natural approach to solving Problem~\eqref{pb} is to apply any efficient method that works in the unconstrained setting and add projections back onto $\mathcal{C}$ when the iterates leave the feasible region. However, there are situations where projections can be very expensive while linear minimizations over $\mathcal{C}$ are much cheaper. For example, if $\mathcal{C}=\{X\in\mathbb{R}^{m\times n}\mid\|X\|_{\operatorname{nuc}}\leq\tau\}$ is a nuclear norm-ball, a projection onto $\mathcal{C}$ requires computing an SVD, which has complexity $\mathcal{O}(mn\min\{m,n\})$, while a linear minimization over $\mathcal{C}$ requires only computing the pair of top singular vectors, which has complexity $\mathcal{O}(\operatorname{nnz})$ where $\operatorname{nnz}$ denotes the number of nonzero entries. Other examples include the flow polytope, the Birkhoff polytope, the matroid polytope, or the set of rotations; see, e.g., \citet{hazan12}.\\

In these situations, the Frank-Wolfe algorithm (FW) \citep{fw56}, a.k.a.~conditional gradient algorithm \citep{polyak66cg}, becomes the method of choice, as it is a simple projection-free algorithm relying on a linear minimization oracle over $\mathcal{C}$. At each iteration, it calls the oracle $v_t\leftarrow\argmin_{v\in\mathcal{C}}\langle\nabla f(x_t),v\rangle$ and moves in the direction of this vertex, ensuring that the new iterate $x_{t+1}\leftarrow x_t+\gamma_t(v_t-x_t)$ is feasible by convex combination, with a step-size $\gamma_t\in\left[0,1\right]$. Hence, FW can be seen as a projection-free variant of projected gradient descent trading the gradient descent direction $-\nabla f(x_t)$ for the vertex direction $v_t-x_t$ minimizing the linear approximation of $f$ at $x_t$ over $\mathcal{C}$. FW has been applied to traffic assignment problems \citep{leblanc75}, low-rank matrix approximation \citep{shalev11rank}, structural SVMs \citep{lacoste13svm}, video co-localization \citep{joulin15video}, infinite RBMs \citep{ping16rbm}, and, e.g., adversarial learning \citep{chen20adversarial}.\\

The main drawback of FW is that the modified descent direction leads to a sublinear convergence rate $\mathcal{O}(1/t)$, which cannot be improved upon in general as an asymptotic lower bound $\Omega(1/t^{1+\delta})$ holds for any $\delta>0$ \citep{canon68}. More recently, \citet{jaggi13fw} provided a simple illustration of the phenomenon: if $f:x\in\mathbb{R}^n\mapsto\|x\|_2^2$ is the squared $\ell_2$-norm and $\mathcal{C}=\Delta_n$ is the standard simplex, then the primal gap at iteration $t\in\llbracket1,n\rrbracket$ is lower bounded by $1/t-1/n$; see also \citet{lan13complex} for a lower bound $\Omega(LD^2/t)$ on an equivalent setup, exhibiting an explicit dependence on the smoothness constant $L$ of $f$ and the diameter $D$ of $\mathcal{C}$.\\

Hence, a vast literature has been devoted to the analysis of higher convergence rates of FW if additional assumptions on the properties of $f$, the geometry of $\mathcal{C}$, or the location of $\argmin_\mathcal{C}f$ are met. Important contributions include:
\begin{enumerate}[label=(\roman*)]
 \item $\mathcal{O}(e^{-\omega t})$ if $\mathcal{C}$ is strongly convex and $\inf_\mathcal{C}\|\nabla f\|>0$ \citep{polyak66cg},
 \item $\mathcal{O}(e^{-\omega t})$ if $f$ is strongly convex and $\argmin_\mathcal{C}f\subset\operatorname{relint}(\mathcal{C})$ \citep{gm86},
 \item $\mathcal{O}(1/t^2)$ if $f$ is gradient dominated and $\mathcal{C}$ is strongly convex \citep{garber15faster}.
\end{enumerate}

More recently, several variants to FW have been proposed, achieving linear convergence rates without excessively increasing the per-iteration complexity. These include the following:
\begin{enumerate}[label=(\roman*)]
 \item $\mathcal{O}(e^{-\omega t})$ when $f$ is strongly convex and $\mathcal{C}$ is a polytope \citep{garber16,lacoste15linear,pok18bcg},
 \item $\mathcal{O}(e^{-\omega t})$ with constants depending on the sparsity of the solution when $f$ is strongly convex and $\mathcal{C}$ is a polytope, of the form $\{x\in\mathbb{R}^n\mid Ax=b,x\geq0\}$ with vertices in $\{0,1\}^n$ \citep{garber16dicg}, or of arbitrary form \citep{bashiri17dicg}.
\end{enumerate}

\paragraph{Contributions.} We propose the \emph{Boosted Frank-Wolfe} algorithm (BoostFW), a new and intuitive method speeding up the Frank-Wolfe algorithm by chasing the negative gradient direction $-\nabla f(x_t)$ via a matching pursuit-style subroutine, and moving in this better aligned direction. BoostFW thereby mimics gradient descent while remaining projection-free. We derive convergence rates $\mathcal{O}(1/t)$ to $\mathcal{O}(e^{-\omega t})$. Although the linear minimization oracle may be called multiple times per iteration, we demonstrate in a series of computational experiments the competitive advantage both per iteration and in CPU time of our method over the state-of-the-art. Furthermore, BoostFW does not require line search to achieve strong empirical performance, and it does not need to maintain the decomposition of the iterates. Naturally, our approach can also be used to boost the performance of any Frank-Wolfe-style algorithm.

\paragraph{Outline.} We start with notation and definitions and we present some background material on the Frank-Wolfe algorithm (Section~\ref{sec:prelim}). We then move on to the intuition behind the design of the Boosted Frank-Wolfe algorithm and present its convergence analysis (Section~\ref{sec:boofw}). We validate the advantage of our approach in a series of computational experiments (Section~\ref{sec:exp}). Finally, a couple of remarks conclude the paper (Section~\ref{sec:conclusion}). All proofs are available in Appendix~\ref{apx:proofs}. The Appendix further contains complementary plots (Appendix~\ref{apx:plots}), an application of our approach to boost the Decomposition-Invariant Pairwise Conditional Gradient algorithm (DICG) \citep{garber16dicg} (Appendix~\ref{apx:boostdicg}), and the convergence analysis of the line search-free Away-Step Frank-Wolfe algorithm (Appendix~\ref{apx:afw}). We were later informed that the latter analysis was already derived by \citet{pedregosa20} in a more general setting.

\section{Preliminaries}
\label{sec:prelim}

We work in a Euclidean space $(\mathcal{H},\langle\cdot,\cdot\rangle)$ equipped with the induced norm $\|\cdot\|$. Let $\mathcal{C}\subset\mathcal{H}$ be a nonempty compact convex set. If $\mathcal{C}$ is a polytope, let $\mathcal{V}$ be its set of vertices. Else, slightly abusing notation, we refer to any point in $\mathcal{V}\coloneqq\partial\mathcal{C}$ as a \emph{vertex}. We denote by $D\coloneqq\max_{x,y\in\mathcal{C}}\|y-x\|$ the diameter of $\mathcal{C}$.

\subsection{Notation and definitions} 

For any $i,j\in\mathbb{N}$ satisfying $i\leq j$, the brackets $\llbracket i,j\rrbracket$ denote the set of integers between (and including) $i$ and $j$. The indicator function for an event $A$ is $\mathds{1}_A\coloneqq1\text{ if }A\text{ is true else }0$. For any $x\in\mathbb{R}^n$ and $i\in\llbracket1,n\rrbracket$, $[x]_i$ denotes the $i$-th entry of $x$. Given $p\geq1$, the $\ell_p$-norm in $\mathbb{R}^n$ is $\|\cdot\|_p:x\in\mathbb{R}^n\mapsto(\sum_{i=1}^n|[x]_i|^p)^{1/p}$ and the closed $\ell_p$-ball of radius $\tau>0$ is $\mathcal{B}_p(\tau)\coloneqq\{x\in\mathbb{R}^n\mid\|x\|_p\leq\tau\}$. The standard simplex in $\mathbb{R}^n$ is $\Delta_n\coloneqq\{x\in\mathbb{R}^n\mid1^\top x=1,x\geq0\}=\operatorname{conv}(e_1,\ldots,e_n)$ where $\{e_1,\ldots,e_n\}$ denotes the standard basis, i.e., $e_i=(\mathds{1}_{\{1=i\}},\ldots,\mathds{1}_{\{n=i\}})^\top$. The conical hull of a nonempty set $\mathcal{A}\subseteq\mathcal{H}$ is $\operatorname{cone}(\mathcal{A})\coloneqq\{\sum_{k=1}^K\lambda_ka_k\mid K\in\mathbb{N}\backslash\{0\},\lambda_1,\ldots,\lambda_K\geq0,a_1,\ldots,a_K\in\mathcal{A}\}$. The number of its elements is denoted by $|\mathcal{A}|$.\\

Let $f:\mathcal{H}\rightarrow\mathbb{R}$ be a differentiable function. We say that $f$ is:
\begin{enumerate}[label=(\roman*)]
 \item $L$-smooth if $L>0$ and for all $x,y\in\mathcal{H}$,
\begin{align*}
 f(y)-f(x)-\langle\nabla f(x),y-x\rangle\leq\frac{L}{2}\|y-x\|^2,
\end{align*}
\item $S$-strongly convex if $S>0$ and for all $x,y\in\mathcal{H}$,
\begin{align*}
 f(y)-f(x)-\langle\nabla f(x),y-x\rangle\geq\frac{S}{2}\|y-x\|^2,
\end{align*}
\item\label{def:pl} $\mu$-gradient dominated if $\mu>0$, $\argmin_{\mathcal{H}}f\neq\varnothing$, and for all $x\in\mathcal{H}$,
\begin{align*}
 f(x)-\min_{\mathcal{H}}f
 \leq\frac{\|\nabla f(x)\|^2}{2\mu}.
\end{align*}
\end{enumerate}

Note that although Definition~\ref{def:pl} is defined with respect to the global optimal value $\min_\mathcal{H}f$, the bound holds for the primal gap of $f$ on any compact set $\mathcal{C}\subset\mathcal{H}$:
\begin{align*}
 f(x)-\min_\mathcal{C}f
 \leq f(x)-\min_\mathcal{H}f
 \leq\frac{\|\nabla f(x)\|^2}{2\mu}.
\end{align*}
Definition~\ref{def:pl} is also commonly referred to as the Polyak-\L ojasiewicz inequality or PL inequality \citep{polyak63pl,lojo63}. It is a \emph{local} condition, weaker than that of strong convexity (Fact~\ref{fact:mu}), but it can still provide linear convergence rates for non-strongly convex functions \citep{karimi16pl}. For example, the least squares loss $x\in\mathbb{R}^n\mapsto\|Ax-b\|_2^2$ where $A\in\mathbb{R}^{m\times n}$ and $\operatorname{rank}(A)=m<n$ is not strongly convex, however it is gradient dominated \citep{garber15faster}. See also the Kurdyka-\L ojasiewicz inequality \citep{kurdyka98,lojo63} for a generalization to nonsmooth optimization \citep{bolte15kl}.

\begin{fact}
 \label{fact:mu}
Let $f:\mathcal{H}\rightarrow\mathbb{R}$ be $S$-strongly convex. Then $f$ is $S$-gradient dominated.
\end{fact}

\subsection{The Frank-Wolfe algorithm}
\label{sec:fw}

The Frank-Wolfe algorithm (FW) \citep{fw56}, a.k.a.~conditional gradient algorithm \citep{polyak66cg}, is presented in Algorithm~\ref{fw}. It is a simple first-order projection-free algorithm relying on a linear minimization oracle over $\mathcal{C}$. At each iteration, it minimizes over $\mathcal{C}$ the linear approximation of $f$ at $x_t$, i.e., $\ell_f(x_t):z\in\mathcal{C}\mapsto f(x_t)+\langle\nabla f(x_t),z-x_t\rangle$, by calling the oracle (Line~\ref{fw:lmo}) and moves in that direction by convex combination (Line~\ref{fw:new}). Hence, the new iterate $x_{t+1}$ is guaranteed to be feasible by convexity and there is no need to use projections back onto $\mathcal{C}$. In short, FW solves Problem~\eqref{pb} by minimizing a sequence of linear approximations of $f$ over $\mathcal{C}$.\\

\begin{algorithm}[h]
\caption{Frank-Wolfe (FW)}
\label{fw}
\textbf{Input:} Start point $x_0\in\mathcal{C}$, step-size strategy $\gamma_t\in\left[0,1\right]$.\\
\textbf{Output:} Point $x_T\in\mathcal{C}$.
\textcolor{white}{yolo}
\begin{algorithmic}[1]
\FOR{$t=0$ \textbf{to} $T-1$}
\STATE$v_t\leftarrow\argmin\limits_{v\in\mathcal{V}}\langle\nabla f(x_t),v\rangle$\hfill$\triangleright${ FW oracle}\label{fw:lmo}
\STATE$x_{t+1}\leftarrow x_t+\gamma_t(v_t-x_t)$\label{fw:new}
\ENDFOR
\end{algorithmic}
\end{algorithm}

Note that FW has access to the feasible region $\mathcal{C}$ only via the linear minimization oracle, which receives any $c\in\mathcal{H}$ as input and outputs a point $v\in\argmin_{z\in\mathcal{C}}\langle c,z\rangle=\argmin_{v\in\mathcal{V}}\langle c,v\rangle$. For example, if $\mathcal{H}=\mathbb{R}^n$ and $\mathcal{C}=\{x\in\mathbb{R}^n\mid\|x\|_1\leq\tau\}$ is an $\ell_1$-ball, then $\mathcal{V}=\{\pm\tau e_1,\ldots,\pm\tau e_n\}$ so the linear minimization oracle simply picks the coordinate $e_i$ with the largest absolute magnitude $|[c]_i|$ and returns $-\operatorname{sign}([c]_i)\tau e_i$. In this case, FW accesses $\mathcal{C}$ only by reading coordinates. Some other examples are covered in the experiments (Section~\ref{sec:exp}).\\

The general convergence rate of FW is $\mathcal{O}(LD^2/t)$, where $L$ is the smoothness constant of $f$ and $D$ is the diameter of $\mathcal{C}$ \citep{polyak66cg,jaggi13fw}. There are different step-size strategies possible to achieve this rate. The default strategy is $\gamma_t\leftarrow2/(t+2)$. It is very simple to implement but it does not guarantee progress at each iteration. The next strategy, sometimes referred to as the \emph{short step} strategy and which does make FW a descent algorithm, is $\gamma_t\leftarrow\min\{\langle\nabla f(x_t),x_t-v_t\rangle/(L\|x_t-v_t\|^2),1\}$. It minimizes the quadratic smoothness upper bound on $f$. If $\epsilon_t\coloneqq f(x_t)-\min_\mathcal{C}f$ denotes the primal gap, then
\begin{align*}
 \epsilon_{t+1}
 &\leq
 \begin{cases}
  \epsilon_t-\displaystyle\frac{\langle\nabla f(x_t),x_t-v_t\rangle^2}{2L\|x_t-v_t\|^2}&\text{if }\gamma_t<1\\
  \epsilon_t/2&\text{if }\gamma_t=1.
 \end{cases}
\end{align*}
As we can already see here, a quadratic improvement in progress is obtained if the direction $v_t-x_t$ in which FW moves is better aligned with that of the negative gradient $-\nabla f(x_t)$. The third step-size strategy is a line search $\gamma_t\leftarrow\argmin_{\gamma\in\left[0,1\right]}f(x_t+\gamma(v_t-x_t))$. It is the most expensive strategy but it does not require (approximate) knowledge of $L$ and it often yields more progress per iteration in practice.

\section{Boosting Frank-Wolfe}
\label{sec:boofw}

\subsection{Motivation}
\label{sec:idea}

Suppose that $\mathcal{C}$ is a polytope and 
that the set of global minimizers $\argmin_\mathcal{H}f$ lies on a lower dimensional face. Then FW can be very slow to converge as it is allowed only to follow vertex directions. As a simple illustration, consider the problem of minimizing $f:x\in\mathbb{R}^2\mapsto\|x\|_2^2/2$ over the convex hull of $\{(-1,0)^\top,(1,0)^\top,(0,1)^\top\}$, starting from $x_0=(0,1)^\top$. The minimizer is $x^*=(0,0)^\top$. We computed the first iterates of FW and we present their trajectory in Figure~\ref{zigzag}. We can see that the iterates try to reach $x^*$ by moving \emph{towards} vertices but clearly these directions $v_t-x_t$ are inadequate as they become orthogonal to $x^*-x_t$.\\

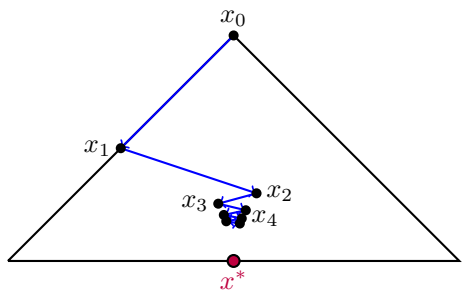
\begin{figure}[h]
 \centering
  \begin{tikzpicture}[thick, scale=1.5]
  \coordinate (a) at (-2, 0);
  \coordinate (b) at (2, 0);
  \coordinate (c) at (0, 2);
  \coordinate (x1) at (-1, 1);
  \coordinate (x2) at (0.203, 0.600);
  \coordinate (x3) at (-0.136, 0.508);
  \coordinate (x4) at (0.108, 0.4497);
  \coordinate (x5) at (-0.0867, 0.408);
  \coordinate (x6) at (0.072, 0.377);
  \coordinate (x7) at (-0.065, 0.352);
  \coordinate (x8) at (0.055, 0.332);
  \coordinate (xx) at (0, 0);
  
  \draw (a) -- (b) -- (c) -- (a);
  
  \draw[->][blue] (c) -- (x1);
  \draw[->][blue] (x1) -- (x2);
  \draw[->][blue] (x2) -- (x3);
  \draw[->][blue] (x3) -- (x4);
  \draw[->][blue] (x4) -- (x5);
  \draw[->][blue] (x5) -- (x6);
  \draw[->][blue] (x6) -- (x7);
  \draw[->][blue] (x7) -- (x8);
  
  \draw[fill=purple] (xx) circle (1.5pt) node [below] {\textcolor{purple}{$x^*$}};
  \draw[fill=black] (c) circle (1pt) node [above] {$x_0$};
  \draw[fill=black] (x1) circle (1pt) node [left] {$x_1$};
  \draw[fill=black] (x2) circle (1pt) node [right] {$x_2$};
  \draw[fill=black] (x3) circle (1pt) node [left] {$x_3$};
  \draw[fill=black] (x4) circle (1pt);
  \draw[fill=black] (x5) circle (1pt);
  \draw[fill=black] (x6) circle (1pt) node [right] {$x_4$};
  \draw[fill=black] (x7) circle (1pt);
  \draw[fill=black] (x8) circle (1pt);
 \end{tikzpicture}
 \caption{FW yields an inefficient zig-zagging trajectory towards the minimizer.}
 \label{zigzag}
\end{figure}

To remedy this phenomenon, \citet{wolfe70} proposed the Away-Step Frank-Wolfe algorithm (AFW), a variant of FW that allows to move \emph{away} from vertices. The issue in Figure~\ref{zigzag} is that the iterates are held back by the weight of vertex $x_0$ in their convex decomposition. Figure~\ref{zigzag_away} shows that AFW is able to remove this weight and thereby to converge much faster to $x^*$. In fact, \citet{lacoste15linear} established that AFW with line search converges at a linear rate $\mathcal{O}\big(LD^2\exp(-(S/(8L))(W/D)^2t)\big)$ for $S$-strongly convex functions over polytopes, where $W$ is the \emph{pyramidal width} of the polytope.\\

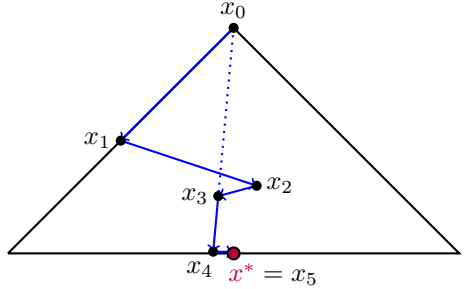
\begin{figure}[h]
\centering
 \begin{tikzpicture}[thick, scale=1.5]
  \coordinate (a) at (-2, 0);
  \coordinate (b) at (2, 0);
  \coordinate (c) at (0, 2);
  \coordinate (x1) at (-1, 1);
  \coordinate (x2) at (0.204, 0.600);
  \coordinate (x3) at (-0.136, 0.508);
  \coordinate (x4) at (-0.181, 0.015);
  \coordinate (x5) at (-0.0022, 0.014);
  \coordinate (xx) at (0, 0);
  
  \draw[dotted][blue] (c) -- (x3);
  
  \draw (a) -- (b) -- (c) -- (a);
  
  \draw[->][blue] (c) -- (x1);
  \draw[->][blue] (x1) -- (x2);
  \draw[->][blue] (x2) -- (x3);
  \draw[->][blue] (x3) -- (x4);
  \draw[->][blue] (x4) -- (x5);
  
  \draw[fill=black] (c) circle (1pt) node [above] {$x_0$};
  \draw[fill=black] (x1) circle (1pt) node [left] {$x_1$};
  \draw[fill=black] (x2) circle (1pt) node [right] {$x_2$};
  \draw[fill=black] (x3) circle (1pt) node [left] {$x_3$};
  \draw[fill=black] (x4) circle (1pt) node [below] {$x_4\quad$};
  \draw[fill=purple] (xx) circle (1.5pt) node [below] {$\quad\quad\quad\textcolor{purple}{x^*}=x_5$};
 \end{tikzpicture}
 \caption{AFW breaks the zig-zagging trajectory by performing away steps. Here, $x_4$ is obtained using an away step which enables $x_5=x^*$, speeding up the algorithm considerably.}
  \label{zigzag_away}
\end{figure}

However, these descent directions are still not as favorable as those of gradient descent, the pyramidal width is a dimension-dependent quantity, and AFW further requires to maintain the decomposition of the iterates onto $\mathcal{V}$ which can become very expensive both in memory usage and computation time \citep{garber16dicg}. Thus, we aim at improving the FW descent direction by directly estimating the gradient descent direction $-\nabla f(x_t)$ using $\mathcal{V}$, in order to maintain the projection-free property. Suppose that $-\nabla f(x_t)\in\operatorname{cone}(\mathcal{V}-x_t)$ and that we are able to compute its conical decomposition, i.e., we have $-\nabla f(x_t)=\sum_{k=0}^{K_t-1}\lambda_k(v_k-x_t)$ where $\lambda_0,\ldots,\lambda_{K_t-1}>0$ and $v_0,\ldots,v_{K_t-1}\in\mathcal{V}$. Then by normalizing by $\Lambda_t\coloneqq\sum_{k=0}^{K_t-1}\lambda_k$, we obtain a \emph{feasible} descent direction $g_t\coloneqq(1/\Lambda_t)\sum_{k=0}^{K_t-1}\lambda_k(v_k-x_t)$ in the sense that $[x_t,x_t+g_t]\subseteq\mathcal{C}$. Therefore, building $x_{t+1}$ as a convex combination of $x_t$ and $x_t+g_t$ ensures that $x_{t+1}\in\mathcal{C}$ and the projection-free property holds as in a typical FW step, all the while moving in the direction of the negative gradient $-\nabla f(x_t)$.

\subsection{Boosting via gradient pursuit}
\label{sec:boosting}

In practice however, computing the exact conical decomposition of $-\nabla f(x_t)$, even when this is feasible, is not necessary and it may be overkill. Indeed, all we want is to find a descent direction $g_t$ using $\mathcal{V}$ that is \emph{better} aligned with $-\nabla f(x_t)$ and we do not mind if $\|-\nabla f(x_t)-g_t\|$ is arbitrarily large. Thus, we propose to chase the direction of $-\nabla f(x_t)$ by sequentially picking up vertices in a matching pursuit-style \citep{mallat93mp}. The procedure is described in Algorithm~\ref{boofw} (Lines~\ref{proc1}-\ref{proc2}). In fact, it implicitly addresses the cone constrained quadratic optimization subproblem
\begin{align}
 \min_{d\in\operatorname{cone}(\mathcal{V}-x_t)}\frac{1}{2}\|-\nabla f(x_t)-d\|^2\label{conicpb}
\end{align}
via the Non-Negative Matching Pursuit algorithm (NNMP) \citep{locatello17conic}, without however the aim of solving it. At each round $k$, the procedure looks to reduce the residual $r_k$ by subtracting its projection $\lambda_ku_k$ onto the principal component $u_k$. The comparison $\langle r_k,v_k-x_t\rangle$ vs.~$\langle r_k,-d_k/\|d_k\|\rangle$ in Line~\ref{boofw:extra} is less intuitive than the rest of the procedure but it is necessary to ensure convergence; see \citet{locatello17conic}. The normalization in Line~\ref{norma} ensures the feasibility of the new iterate $x_{t+1}$.\\

\begin{algorithm}[h]
\caption{Boosted Frank-Wolfe (BoostFW)}
\label{boofw}
\textbf{Input:} Input point $y\in\mathcal{C}$, maximum number of rounds $K\in\mathbb{N}\backslash\{0\}$, alignment improvement tolerance $\delta\in\left]0,1\right[$, step-size strategy $\gamma_t\in\left[0,1\right]$.\\
\textbf{Output:} Point $x_T\in\mathcal{C}$.
\textcolor{white}{yolo}
\begin{algorithmic}[1]
\STATE$x_0\leftarrow\argmin\limits_{v\in\mathcal{V}}\langle\nabla f(y),v\rangle$\label{boofw:x0}
\FOR{$t=0$ \textbf{to} $T-1$}
\STATE$d_0\leftarrow0$\label{proc1}
\STATE$\Lambda_t\leftarrow0$
\STATE$\text{flag}\leftarrow\textbf{false}$
\FOR{$k=0$ \textbf{to} $K-1$}
\STATE$r_k\leftarrow-\nabla f(x_t)-d_k$\hfill$\triangleright${ $k$-th residual}
\STATE$v_k\leftarrow\argmax\limits_{v\in\mathcal{V}}\langle r_k,v\rangle$\hfill$\triangleright${ FW oracle}\label{boofw:v}
\STATE$u_k\leftarrow\argmax\limits_{u\in\{v_k-x_t,-d_k/\|d_k\|\}}\langle r_k,u\rangle$\label{boofw:extra}
\STATE$\lambda_k\leftarrow\displaystyle\frac{\langle r_k,u_k\rangle}{\|u_k\|^2}$\label{boofw:lambda}
\STATE$d_k'\leftarrow d_k+\lambda_ku_k$\label{boofw:dprime}
\IF{$\operatorname{align}(-\nabla f(x_t),d_{k}')-\operatorname{align}(-\nabla f(x_t),d_k)\geq\delta$}\label{criterion}
\STATE$d_{k+1}\leftarrow d_k'$
\STATE$\Lambda_t\leftarrow\begin{cases}\Lambda_t+\lambda_k&\textbf{if }u_k=v_k-x_t\\\Lambda_t(1-\lambda_k/\|d_k\|)&\textbf{if }u_k=-d_k/\|d_k\|\end{cases}$\label{boofw:Lbd}
\ELSE\label{else}
\STATE$\text{flag}\leftarrow\textbf{true}$
\STATE\textbf{break}\hfill$\triangleright${ exit $k$-loop}\label{break}
\ENDIF
\ENDFOR\label{proc2}
\STATE$K_t\leftarrow k$ \textbf{if} $\text{flag}=\textbf{true}$ \textbf{else} $K$\label{boofw:kt}
\STATE$g_t\leftarrow d_{K_t}/\Lambda_t$\hfill$\triangleright${ normalization}\label{norma}
\STATE$x_{t+1}\leftarrow x_t+\gamma_tg_t$\label{update}
\ENDFOR
\end{algorithmic}
\end{algorithm}

Since we are only interested in the direction of $-\nabla f(x_t)$, the stopping criterion in the procedure (Line~\ref{criterion}) is an alignment condition between $-\nabla f(x_t)$ and the current estimated direction $d_k$, which serves as descent direction for BoostFW. The function $\operatorname{align}$, defined in~\eqref{angle}, measures the alignment between a target direction $d\in\mathcal{H}\backslash\{0\}$ and its estimate $\hat{d}\in\mathcal{H}$. It is invariant by scaling of $d$ or $\hat{d}$, and the higher the value, the better the alignment:
\begin{align}
 \operatorname{align}(d,\hat{d})\coloneqq
 \begin{cases}
  \displaystyle\frac{\langle d,\hat{d}\rangle}{\|d\|\|\hat{d}\|}&\text{if }\hat{d}\neq0\\
  -1&\text{if }\hat{d}=0.
 \end{cases}\label{angle}
\end{align}
In order to optimize the trade-off between progress and complexity per iteration, we allow for (very) inexact alignments and we stop the procedure as soon as \emph{sufficient} progress is not met (Lines~\ref{else}-\ref{break}). Furthermore, note that it is not possible to obtain a perfect alignment when $-\nabla f(x_t)\notin\operatorname{cone}(\mathcal{V}-x_t)$, but this is not an issue as we only seek to better align the descent direction. The number of pursuit rounds at iteration $t$ is denoted by $K_t$ (Line~\ref{boofw:kt}). In the experiments (Section~\ref{sec:exp}), we typically set $\delta\leftarrow10^{-3}$ and $K\leftarrow+\infty$; the role of $K$ is only to cap the number of pursuit rounds per iteration when the FW oracle is particularly expensive (see Section~\ref{sec:traffic}). Note that if $K=1$ then BoostFW reduces to FW.\\

In the case of Figures~\ref{zigzag}-\ref{zigzag_away}, BoostFW exactly estimates the direction of $-\nabla f(x_0)=-(x_0-x^*)$ in only two rounds and converges in $1$ iteration. A more general illustration of the procedure is presented in Figure~\ref{fig:procedure}. See also Appendix~\ref{apx:angle-improvs} for an illustration of the improvements in alignment of $d_k$ during the procedure. Lastly, note that BoostFW does not need to maintain the decomposition of the iterates, which is very favorable in practice \citep{garber16dicg}.\\

\begin{figure}[h]
 \begin{center}
 \begin{tikzpicture}[thick, scale=1.3]
  \coordinate (a) at (-1.2, 1);
  \coordinate (b) at (-0.2, 2);
  \coordinate (c) at (2, 0.8);
  \coordinate (d) at (1.3, -0.8);
  \coordinate (e) at (-0.8, -0.5);
  \coordinate (f) at (-1.2, 1);
  \coordinate (x) at (0.6, 0.6);
  \coordinate (n) at (1.5, 1.5);
  \coordinate (g) at (1.082, 1.172);
  \coordinate (r1s) at (1.608, 0.744);
  \coordinate (r1e) at (1.5, 1.5);
  \coordinate (l0s) at (0.6, 0.6);
  \coordinate (l0e) at (1.608, 0.744);
  \coordinate (r1bs) at (0.6, 0.6);
  \coordinate (r1be) at (0.492, 1.356);
  \coordinate (r2s) at (0.247, 1.216);
  \coordinate (r2e) at (0.492, 1.356);
  \coordinate (l1s) at (0.6, 0.6);
  \coordinate (l1e) at (0.247, 1.216);
  
  \node at (2.1, -0.7) {\emph{(a)}};
  
  \draw[dotted][purple] (l0e) -- (c);
  
  \draw (a) -- (b) -- (c) -- (d) -- (e) -- (f) -- (a);
  
  \draw[->][purple] (r1s) -- (r1e) node [midway, right] {\textcolor{purple}{$r_1$}};
  \draw[->][purple] (l0s) -- (l0e) node [below] {\textcolor{purple}{$\lambda_0u_0\quad\;$}};
  \draw[->] (x) -- (n);
  
  \draw[fill=black] (c) circle (1pt) node [right] {$v_0$};
  \draw[fill=black] (x) circle (1pt) node [below] {$x_t$};
  \draw (n) node [above] {$\quad-\nabla f(x_t)=r_0$};
 \end{tikzpicture}\hspace{15mm}
  \begin{tikzpicture}[thick, scale=1.3]
  \coordinate (a) at (-1.2, 1);
  \coordinate (b) at (-0.2, 2);
  \coordinate (c) at (2, 0.8);
  \coordinate (d) at (1.3, -0.8);
  \coordinate (e) at (-0.8, -0.5);
  \coordinate (f) at (-1.2, 1);
  \coordinate (x) at (0.6, 0.6);
  \coordinate (n) at (1.5, 1.5);
  \coordinate (g) at (1.082, 1.172);
  \coordinate (r1s) at (1.608, 0.744);
  \coordinate (r1e) at (1.5, 1.5);
  \coordinate (l0s) at (0.6, 0.6);
  \coordinate (l0e) at (1.608, 0.744);
  \coordinate (r1bs) at (0.6, 0.6);
  \coordinate (r1be) at (0.492, 1.356);
  \coordinate (r2s) at (0.247, 1.216);
  \coordinate (r2e) at (0.492, 1.356);
  \coordinate (l1s) at (0.6, 0.6);
  \coordinate (l1e) at (0.247, 1.216);
  
  \node at (2.1, -0.7) {\emph{(b)}};
  
  \draw[dotted][teal] (l1e) -- (b);
  \draw[dotted][purple] (l0e) -- (c);
  
  \draw (a) -- (b) -- (c) -- (d) -- (e) -- (f) -- (a);
  
  \draw[->][purple] (l0s) -- (l0e) node [below] {\textcolor{purple}{$\lambda_0u_0\quad\;$}};
  \draw[->][purple] (r1bs) -- (r1be) node [midway, right] {\textcolor{purple}{$r_1$}};
  \draw[->][teal] (r2s) -- (r2e) node [midway, above] {\textcolor{teal}{$r_2$}};
  \draw[->][teal] (l1s) -- (l1e) node [midway, left] {\textcolor{teal}{$\lambda_1u_1$}};
  
  \draw[fill=black] (b) circle (1pt) node [above] {$v_1$};
  \draw[fill=black] (c) circle (1pt) node [right] {$v_0$};
  \draw[fill=black] (x) circle (1pt) node [below] {$x_t$};
 \end{tikzpicture}
 
 \vspace{5mm}
 
 \begin{tikzpicture}[thick, scale=1.3]
  \coordinate (a) at (-1.2, 1);
  \coordinate (b) at (-0.2, 2);
  \coordinate (c) at (2, 0.8);
  \coordinate (d) at (1.3, -0.8);
  \coordinate (e) at (-0.8, -0.5);
  \coordinate (f) at (-1.2, 1);
  \coordinate (x) at (0.6, 0.6);
  \coordinate (n) at (1.5, 1.5);
  \coordinate (g) at (1.082, 1.172);
  \coordinate (r1s) at (1.608, 0.744);
  \coordinate (r1e) at (1.5, 1.5);
  \coordinate (l0s) at (0.6, 0.6);
  \coordinate (l0e) at (1.608, 0.744);
  \coordinate (r1bs) at (0.6, 0.6);
  \coordinate (r1be) at (0.492, 1.356);
  \coordinate (r2s) at (0.247, 1.216);
  \coordinate (r2e) at (0.492, 1.356);
  \coordinate (l1s) at (0.6, 0.6);
  \coordinate (l1e) at (0.247, 1.216);
  \coordinate (d2) at (1.255, 1.36);
  
  \node at (2.1, -0.7) {\emph{(c)}};
  
  \draw[dotted][teal] (l1e) -- (b);
  \draw[dotted][purple] (l0e) -- (c);
  \draw[dotted][purple] (l0e) -- (d2);
  \draw[dotted][teal] (l1e) -- (d2);
  
  \draw (a) -- (b) -- (c) -- (d) -- (e) -- (f) -- (a);
  
  \draw[->][purple] (l0s) -- (l0e) node [below] {\textcolor{purple}{$\lambda_0u_0\quad\;$}};
  \draw[->][teal] (l1s) -- (l1e) node [midway, left] {\textcolor{teal}{$\lambda_1u_1$}};
  \draw[->][orange] (x) -- (d2) node [above] {\textcolor{orange}{$d_2$}};
  
  \draw[fill=black] (b) circle (1pt) node [above] {$v_1$};
  \draw[fill=black] (c) circle (1pt) node [right] {$v_0$};
  \draw[fill=black] (x) circle (1pt) node [below] {$x_t$};
 \end{tikzpicture}\hspace{15mm}
  \begin{tikzpicture}[thick, scale=1.3]
  \coordinate (a) at (-1.2, 1);
  \coordinate (b) at (-0.2, 2);
  \coordinate (c) at (2, 0.8);
  \coordinate (d) at (1.3, -0.8);
  \coordinate (e) at (-0.8, -0.5);
  \coordinate (f) at (-1.2, 1);
  \coordinate (x) at (0.6, 0.6);
  \coordinate (n) at (1.5, 1.5);
  \coordinate (g) at (1.082, 1.172);
  \coordinate (r1s) at (1.608, 0.744);
  \coordinate (r1e) at (1.5, 1.5);
  \coordinate (l0s) at (0.6, 0.6);
  \coordinate (l0e) at (1.608, 0.744);
  \coordinate (r1bs) at (0.6, 0.6);
  \coordinate (r1be) at (0.492, 1.356);
  \coordinate (r2s) at (0.247, 1.216);
  \coordinate (r2e) at (0.492, 1.356);
  \coordinate (l1s) at (0.6, 0.6);
  \coordinate (l1e) at (0.247, 1.216);
  
  \node at (2.1, -0.7) {\emph{(d)}};
  
  \draw[dotted][purple] (x) -- (c);
  
  \draw (a) -- (b) -- (c) -- (d) -- (e) -- (f) -- (a);
  
  \draw[->] (x) -- (n);
  \draw[->][orange] (x) -- (d2) node [above] {\textcolor{orange}{$d_2\quad$}};
  \draw[->][blue] (x) -- (g) node [left] {\textcolor{blue}{$g_t$}};
  
  \draw[fill=black] (c) circle (1pt) node [right] {$v_0$};
  \draw[fill=black] (x) circle (1pt) node [below] {$x_t$};
  \draw (n) node [right] {$-\nabla f(x_t)$};
 \end{tikzpicture}
 \end{center}
 \caption{The gradient pursuit procedure builds a descent direction $g_t$ better aligned with the negative gradient direction $-\nabla f(x_t)$, while the FW descent direction is that of $v_0-x_t$. We have $g_t=d_2/(\lambda_0+\lambda_1)$ where $d_2=\lambda_0u_0+\lambda_1u_1$, $u_0=v_0-x_t$, and $u_1=v_1-x_t$. Furthermore, note that $[x_t,x_t+d_2]\not\subseteq\mathcal{C}$ but $[x_t,x_t+g_t]\subseteq\mathcal{C}$. Moving along the segment $[x_t,x_t+g_t]$ ensures feasibility of the new iterate $x_{t+1}$.}
 \label{fig:procedure}
\end{figure}
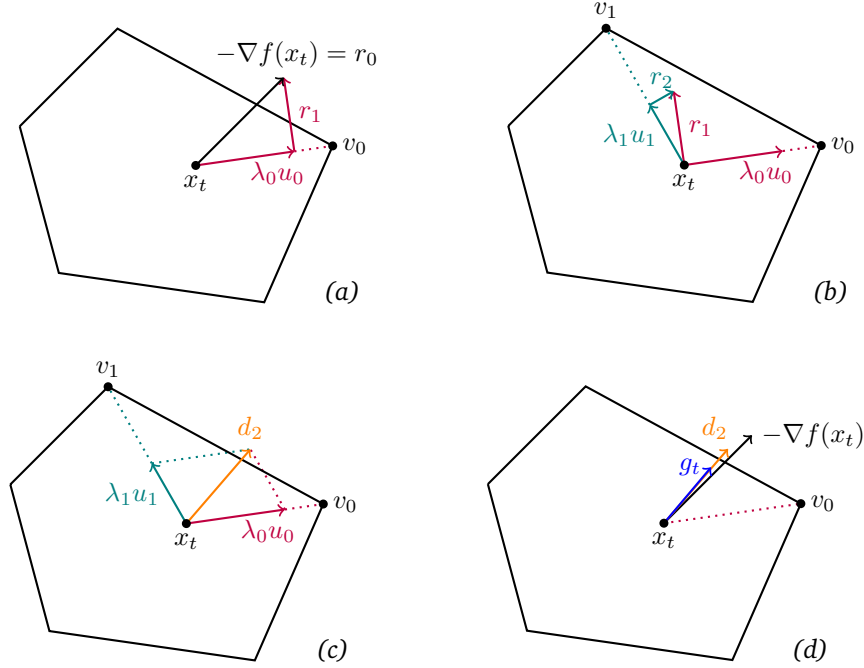

We present in Proposition~\ref{prop} some properties satisfied by BoostFW (Algorithm~\ref{boofw}). Proofs are available in Appendix~\ref{apx:boosting}.

\begin{proposition}
 \label{prop}
 For all $t\in\llbracket0,T-1\rrbracket$,
 \begin{enumerate}[label=(\roman*)]
  \item $d_{1}$ is defined and $K_t\geq1$,
  \item $\lambda_0,\ldots,\lambda_{K_t-1}\geq0$,
  \item $d_k\in\operatorname{cone}(\mathcal{V}-x_t)$ for all $k\in\llbracket0,K_t\rrbracket$,
  \item\label{prop:gx} $x_t+g_t\in\mathcal{C}$ and $x_{t+1}\in\mathcal{C}$,
  \item\label{prop:eta} $\operatorname{align}(-\nabla f(x_t),g_t)\geq\operatorname{align}(-\nabla f(x_t),v_t-x_t)+(K_t-1)\delta$ where $v_t\in\argmin_{v\in\mathcal{V}}\langle\nabla f(x_t),v\rangle$ and $\operatorname{align}(-\nabla f(x_t),v_t-x_t)\geq0$.
 \end{enumerate}
\end{proposition}

\subsection{Convergence analysis}
\label{sec:cv}

We denote by $\eta_t\coloneqq\operatorname{align}(-\nabla f(x_t),g_t)$. We provide in Theorem~\ref{th:boofwsc} the general convergence rate of BoostFW. All proofs are available in Appendix~\ref{apx:cvproofs}. Note that $\eta_t\|\nabla f(x_t)\|/(L\|g_t\|)=\langle-\nabla f(x_t),g_t\rangle/(L\|g_t\|^2)$ corresponds to the short step strategy.

\begin{theorem}[(Universal rate)]
 \label{th:boofwsc}
 Let $f:\mathcal{H}\rightarrow\mathbb{R}$ be $L$-smooth, convex, and $\mu$-gradient dominated, and set $\gamma_t\leftarrow\min\{\eta_t\|\nabla f(x_t)\|/(L\|g_t\|),1\}$ or $\gamma_t\leftarrow\argmin_{\gamma\in\left[0,1\right]}f(x_t+\gamma g_t)$. Then for all $t\in\llbracket0,T\rrbracket$,
 \begin{align*}
  f(x_t)-\min_\mathcal{C}f
  \leq\frac{LD^2}{2}\prod_{s=0}^{t-1}\left(1-\eta_s^2\frac{\mu}{L}\right)^{\mathds{1}_{\{\gamma_s<1\}}}\left(1-\frac{\|g_s\|}{2\|v_s-x_s\|}\right)^{\mathds{1}_{\{\gamma_s=1\}}}
 \end{align*}
 where $v_s\in\argmin_{v\in\mathcal{V}}\langle\nabla f(x_s),v\rangle$ for all $s\in\llbracket0,T-1\rrbracket$.
\end{theorem}

Strictly speaking, the rate in Theorem~\ref{th:boofwsc} is not \emph{explicit} although it still provides a quantitative estimation. Note that $\gamma_t=1$ is extremely rare in practice, and we observed no more than $1$ such iteration in each of the experiments (Section~\ref{sec:exp}). This is a similar phenomenon to that in the Away-Step and Pairwise Frank-Wolfe algorithms \citep{lacoste15linear}. Similarly, $K_t>1$ simply means that it is possible to increase the alignment by $\delta$ twice and consecutively, where $\delta$ is typically set to a low value. In the experiments, we set $\delta\leftarrow10^{-3}$ and we observed $K_t>1$ (or even $K_t>5$) almost everytime.\\

For completeness, we disregard these observations and address in Theorem~\ref{th:1/t} the case where the number of iterations with $\gamma_t<1$ and $K_t>1$ is not dominant, and we add a minor adjustment to Algorithm~\ref{boofw}: if $\gamma_t=1$ then we choose to do a simple FW step, i.e., to move in the direction of $v_{k=0}-x_t$ instead of the direction of $g_t$, where $v_{k=0}$ is computed in the first round of the procedure (Line~\ref{boofw:v}). Although this usually provides less progress, we do it for the sole purpose of presenting a \emph{fully explicit} convergence rate; again, there is no need for such tweaks in practice as typically almost every iteration satisfies $\gamma_t<1$ and $K_t>1$. Theorem~\ref{th:1/t} states the convergence rate for this scenario, which is very loose as it accommodates for these FW steps.

\begin{theorem}[(Worst-case rate)]
 \label{th:1/t}
 Let $f:\mathcal{H}\rightarrow\mathbb{R}$ be $L$-smooth, convex, and $\mu$-gradient dominated, and set $\gamma_t\leftarrow\min\{\eta_t\|\nabla f(x_t)\|/(L\|g_t\|),1\}$ or $\gamma_t\leftarrow\argmin_{\gamma\in\left[0,1\right]}f(x_t+\gamma g_t)$. Consider Algorithm~\ref{boofw} with the minor adjustment $x_{t+1}\leftarrow x_t+\gamma_t'(v_t-x_t)$ in Line~\ref{update} when $\gamma_t=1$, where $v_t\leftarrow v_{k=0}$ is computed in Line~\ref{boofw:v} and $\gamma_t'\leftarrow\min\{\langle\nabla f(x_t),x_t-v_t\rangle/(L\|x_t-v_t\|^2),1\}$ or $\gamma_t'\leftarrow\argmin_{\gamma\in\left[0,1\right]}f(x_t+\gamma(v_t-x_t))$. Then for all $t\in\llbracket0,T\rrbracket$,
 \begin{align*}
  f(x_t)-\min_\mathcal{C}f
  \leq\frac{4LD^2}{t+2}.
 \end{align*}
\end{theorem}

We now provide in Theorem~\ref{th:boofwgood} the more realistic convergence rate of BoostFW, where $N_t\coloneqq|\{s\in\llbracket0,t-1\rrbracket\mid\gamma_s<1,K_s>1\}|$ is \emph{nonnegligeable}, i.e., $N_t\geq\omega t^p$ for some $\omega>0$ and $p\in\left]0,1\right]$. This is the rate observed in practice, where $N_t\approx t-1$ so $\omega\lesssim1$ and $p=1$ (Section~\ref{sec:exp}).

\begin{theorem}[(Practical rate)]
 \label{th:boofwgood} 
 Let $f:\mathcal{H}\rightarrow\mathbb{R}$ be $L$-smooth, convex, and $\mu$-gradient dominated, and set $\gamma_t\leftarrow\min\{\eta_t\|\nabla f(x_t)\|/(L\|g_t\|),1\}$ or $\gamma_t\leftarrow\argmin_{\gamma\in\left[0,1\right]}f(x_t+\gamma g_t)$. Assume that $|\{s\in\llbracket0,t-1\rrbracket\mid\gamma_s<1,K_s>1\}|\geq\omega t^p$ for all $t\in\llbracket 0,T-1\rrbracket$, for some $\omega>0$ and $p\in\left]0,1\right]$. Then for all $t\in\llbracket0,T\rrbracket$,
 \begin{align*}
  f(x_t)-\min_\mathcal{C}f
  \leq\frac{LD^2}{2}\exp\left(-\delta^2\frac{\mu}{L}\omega t^p\right).
 \end{align*}
\end{theorem}

\begin{remark}
Note that when $\gamma_t<1$ and $K_t>1$, we have (see proofs in Appendix~\ref{apx:cvproofs})
 \begin{align*}
  \left(f(x_t)-\min_\mathcal{C}f\right)-\left(f(x_{t+1})-\min_\mathcal{C}f\right)
  \geq\delta^2\frac{\|\nabla f(x_t)\|^2}{2L}
 \end{align*}
 so if $N_T\coloneqq|\{t\in\llbracket0,T-1\rrbracket\mid\gamma_t<1,K_t>1\}|$, then
 \begin{align*}
  f(x_0)-\min_\mathcal{C}f\geq\delta^2\frac{\inf_\mathcal{C}\|\nabla f\|^2}{2L}N_T.
 \end{align*}
 Thus, if $\inf_\mathcal{C}\|\nabla f\|>0$ then 
 \begin{align*}
  N_T
  \leq\frac{2L(f(x_0)-\min_\mathcal{C}f)}{\delta^2\inf_\mathcal{C}\|\nabla f\|^2}
  \leq\left(\frac{LD}{\delta\inf_\mathcal{C}\|\nabla f\|}\right)^2
 \end{align*}
 since $f(x_0)-\min_\mathcal{C}f\leq LD^2/2$ (see proofs in Appendix~\ref{apx:cvproofs}). However, the assumption in Theorem~\ref{th:boofwgood} can still hold as convergence is usually achieved within $T$ iterations where
 \begin{align*}
  T=\mathcal{O}\left(\left(\frac{1}{\omega}\left(\frac{LD}{\delta\inf_\mathcal{C}\|\nabla f\|}\right)^2\right)^{1/p}\right)
 \end{align*}
 for some $\omega>0$ and $p\in\left]0,1\right]$. In the experiments for example (Section~\ref{sec:exp}), convergence is always achieved within $\mathcal{O}(10^3)$ iterations. Furthermore, early stopping to increase the generalization error of a model also prevents $T$ from blowing up.
\end{remark}

Lastly, we provide in Corollary~\ref{cor:lmo} a bound on the number of FW oracle calls, i.e., the number of linear minimizations over $\mathcal{C}$, performed to achieve $\epsilon$-convergence. In comparison, FW and AFW respectively require $\mathcal{O}(LD^2/\epsilon)$ and $\mathcal{O}\big((L/S)(D/W)^2\ln(1/\epsilon)\big)$ oracle calls, where $f$ is assumed to be $S$-strongly convex and $\mathcal{C}$ is assumed to be a polytope with pyramidal width $W$ for AFW \citep{lacoste15linear}. It is clear from its design that BoostFW performs more oracle calls per iteration, however it uses them more efficiently and the progress obtained overcomes the cost. This is demonstrated in the experiments (Section~\ref{sec:exp}).

\begin{corollary}
\label{cor:lmo}
 In order to achieve $\epsilon$-convergence, the number of linear minimizations performed over $\mathcal{C}$ is 
 \begin{align*}
 \begin{cases}
  \displaystyle\mathcal{O}\left(\frac{LD^2\min\{K,1/\delta\}}{\epsilon}\right)&\text{in the worst-case scenario}\\
  \displaystyle\mathcal{O}\left(\min\left\{K,\frac{1}{\delta}\right\}\left(\frac{1}{\omega\delta^{2}}\frac{L}{\mu}\ln\left(\frac{1}{\epsilon}\right)\right)^{1/p}\right)&\text{in the practical scenario}.
  \end{cases}
 \end{align*}
 Note that the practical scenario assumes that we have set $K\geq2$ in BoostFW ($K=1$ reduces BoostFW to FW).
\end{corollary}

\section{Computational experiments}
\label{sec:exp}

We compared the Boosted Frank-Wolfe algorithm (BoostFW, Algorithm~\ref{boofw}) to the Away-Step Frank-Wolfe algorithm (AFW) \citep{wolfe70}, the Decomposition-Invariant Pairwise Conditional Gradient algorithm (DICG) \citep{garber16dicg}, and the Blended Conditional Gradients algorithm (BCG) \citep{pok18bcg} in a series of computational experiments. We ran two strategies for AFW, one with the default line search (AFW-ls) and one using the smoothness of $f$ (AFW-L):
\begin{align*}
 \gamma_t\leftarrow
 \begin{cases}
  \displaystyle\min\left\{\frac{\langle\nabla f(x_t),x_t-v_t^\text{FW}\rangle}{L\|x_t-v_t^\text{FW}\|_2^2},1\right\}&\text{if FW step}\\
  \displaystyle\min\left\{\frac{\langle\nabla f(x_t),v_t^\text{away}-x_t\rangle}{L\|v_t^\text{away}-x_t\|_2^2},\gamma_{\max}\right\}&\text{if away step}
 \end{cases}
\end{align*}
where $\gamma_{\max}$ is defined in the algorithm (see Algorithm~\ref{afw} in Appendix~\ref{apx:afw}). Contrary to common belief, both strategies yield the same linear convergence rate; see \citet{lacoste15linear} for AFW-ls and Theorem~\ref{th:afw} in the Appendix for AFW-L \citep{pedregosa20}. For BoostFW, we also ran a line search strategy to demonstrate that the speed-up really comes from the boosting procedure and not from being line search-free. Results further show that the line search-free strategy $\gamma_t\leftarrow\min\{\eta_t\|\nabla f(x_t)\|/(L\|g_t\|),1\}=\min\{\langle-\nabla f(x_t),g_t\rangle/(L\|g_t\|^2),1\}$ is very performant in CPU time. The line search-free strategy of DICG was not competitive in the experiments.\\

DICG is not applicable to optimization problems over the $\ell_1$-ball
\begin{align}
 \min_{x\in\mathbb{R}^n}\;&f(x)\label{pb:n}\\
 \text{s.t.}\;&\|x\|_1\leq\tau,\nonumber
\end{align}
however we can perform a change of variables $x_i=z_i-z_{n+i}$ and use the following reformulation over the simplex:
\begin{align}
 \min_{z\in\mathbb{R}^{2n}}\;&f([z]_{1:n}-[z]_{n+1:2n})\label{pb:2n}\\
 \text{s.t.}\;&z\in\tau\Delta_{2n}\nonumber
\end{align}
where $[z]_{1:n}$ and $[z]_{n+1:2n}$ denote the truncation to $\mathbb{R}^n$ of the first $n$ entries and the last $n$ entries of $z\in\mathbb{R}^{2n}$ respectively. Fact~\ref{fact:2n} formally states the equivalence between problems~\eqref{pb:n} and~\eqref{pb:2n}. A proof can be found in Appendix~\ref{apx:expproofs}.

\begin{fact}
 \label{fact:2n}
 Consider $\mathbb{R}^n$ and let $\tau>0$. Then $\mathcal{B}_1(\tau)=\{[z]_{1:n}-[z]_{n+1:2n}\mid z\in\tau\Delta_{2n}\}$.
\end{fact}

We implemented all the algorithms in Python using the same code framework for fair comparisons. In the case of synthetic data, we generated them from Gaussian distributions. We ran the experiments on a laptop under Linux Ubuntu 18.04 with Intel Core i7 3.5GHz CPU and 8GB RAM. Code is available at \url{https://github.com/cyrillewcombettes/boostfw}. In each experiment, we estimated the smoothness constant $L$ of the (convex) objective function $f:\mathbb{R}^n\rightarrow\mathbb{R}$, i.e., the Lipschitz constant of the gradient function $\nabla f:\mathbb{R}^n\rightarrow\mathbb{R}^n$, by sampling a few pairs of points $(x,y)\in\mathcal{C}\times\mathcal{C}$ and computing an upper bound on $\|\nabla f(y)-\nabla f(x)\|_2/\|y-x\|_2$. Unless specified otherwise, we set $\delta\leftarrow10^{-3}$ and $K\leftarrow+\infty$ in BoostFW. The role of $K$ is only to cap the number of pursuit rounds per iteration when the FW oracle is particularly expensive (see Section~\ref{sec:traffic}).

\subsection{Sparse signal recovery}
\label{sec:signal}

Let $x^*\in\mathbb{R}^n$ be a signal which we want to recover as a sparse representation from observations $y=Ax^*+w$, where $A\in\mathbb{R}^{m\times n}$ and $w\sim\mathcal{N}(0,\sigma^2I_m)$ is the noise in the measurements. The natural formulation of the problem is
\begin{align*}
 \min_{x\in\mathbb{R}^n}\;&\|y-Ax\|_2^2\\
 \text{s.t.}\;&\|x\|_0\leq\|x^*\|_0
\end{align*}
but the $\ell_0$-pseudo-norm $\|\cdot\|_0:x\in\mathbb{R}^n\mapsto|\{i\in\llbracket1,n\rrbracket\mid[x]_i\neq0\}|$ is nonconvex and renders the problem intractable in many situations \citep{nphard95}. To remedy this, the $\ell_1$-norm is often used as a convex surrogate and leads to the following lasso formulation \citep{tib96lasso} of the problem:
\begin{align*}
 \min_{x\in\mathbb{R}^n}\;&\|y-Ax\|_2^2\\
 \text{s.t.}\;&\|x\|_1\leq\|x^*\|_1.
\end{align*}
In order to compare to DICG, which is not applicable to this formulation, we ran all algorithms on the reformulation~\eqref{pb:2n}. We set $m=200$, $n=500$, $\sigma=0.05$, and $\tau=\|x^*\|_1$. Since the objective function is quadratic, we can derive a closed-form solution to the line search and there is no need for AFW-L or BoostFW-L. The results are presented in Figure~\ref{fig:signal}.\\

\begin{figure}[H]
\centering{\includegraphics[scale=0.59]{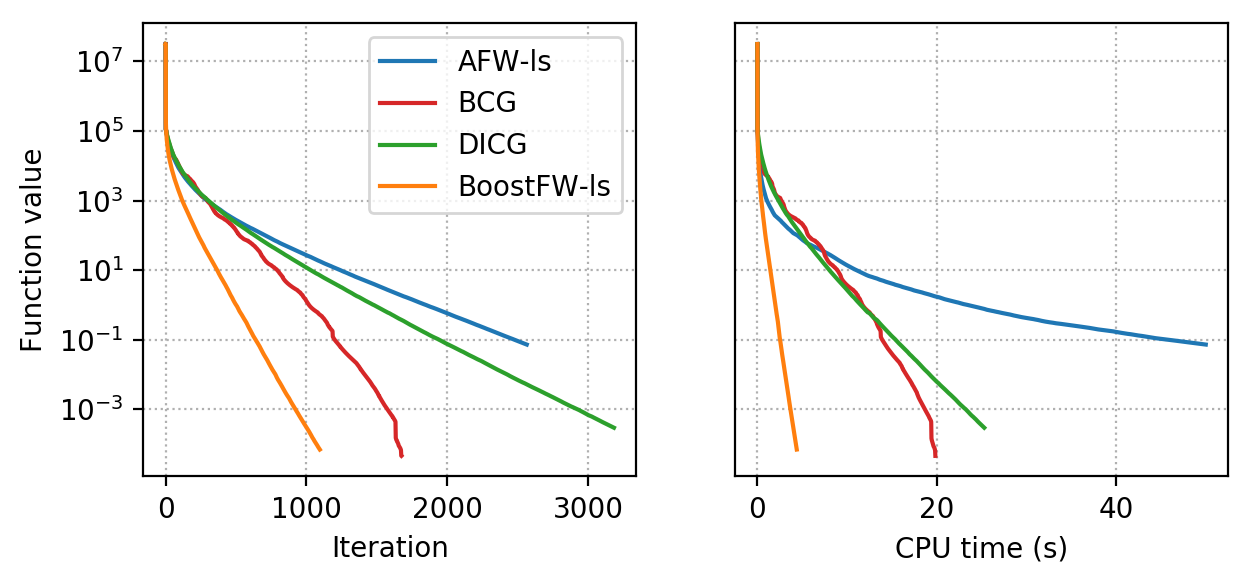}}
\caption{Sparse signal recovery.}
\label{fig:signal}
\end{figure}

\subsection{Sparsity-constrained logistic regression}
\label{sec:gisette}

We consider the task of recognizing the handwritten digits \textsf{4} and \textsf{9} from the Gisette dataset \citep{guyon05challenge}, available at \url{https://archive.ics.uci.edu/ml/datasets/Gisette}. The dataset includes a high number of distractor features with no predictive power. Hence, a sparsity-constrained logistic regression model is suited for the task. The sparsity-inducing constraint is realized via the $\ell_1$-norm:
\begin{align*}
 \min_{x\in\mathbb{R}^n}\;&\frac{1}{m}\sum_{i=1}^m\ln(1+\exp(-y_ia_i^\top x))\\
 \text{s.t.}\;&\|x\|_1\leq\tau
\end{align*}
where $a_1,\ldots,a_m\in\mathbb{R}^{n}$ and $y\in\{-1,+1\}^m$. In order to compare to DICG, which is not applicable to this formulation, we ran all algorithms on the reformulation~\eqref{pb:2n}. We used $m=2000$ samples and the number of features is $n=5000$. We set $\tau=10$, $L=0.5$, and $\delta\leftarrow10^{-4}$ in BoostFW. The results are presented in Figure~\ref{fig:gisette}. As expected, AFW-L and BoostFW-L converge faster in CPU time as they do not rely on line search, however they converge slower per iteration as each iteration provides less progress.\\

\begin{figure}[H]
\centering{\includegraphics[scale=0.59]{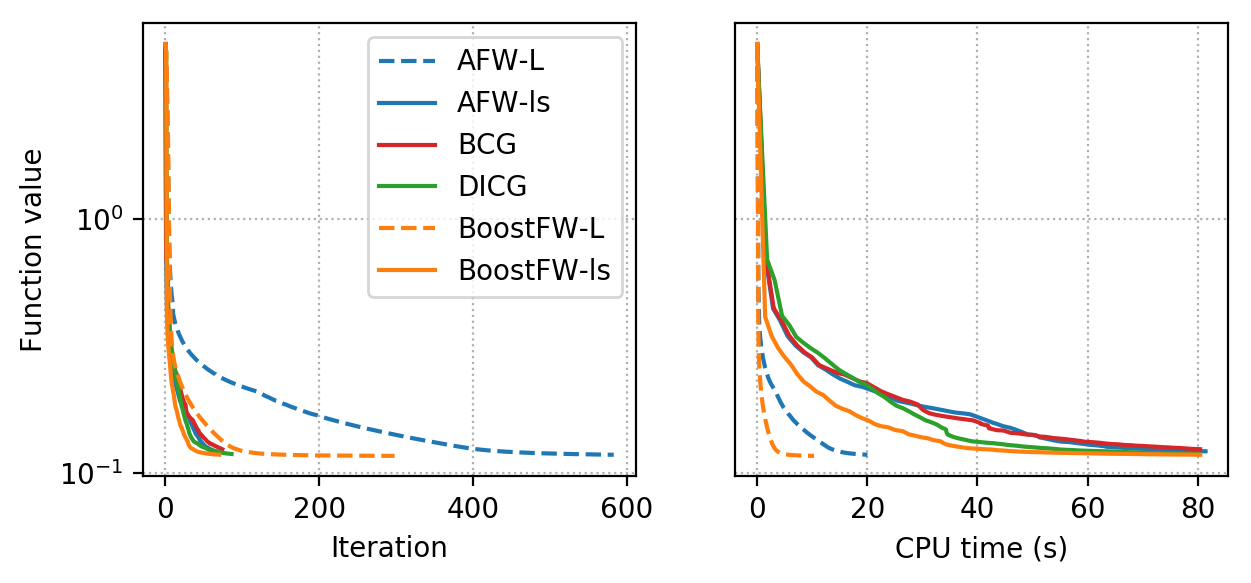}}
\caption{Sparse logistic regression on the Gisette dataset.}
\label{fig:gisette}
\end{figure}

\subsection{Traffic assignment}
\label{sec:traffic}

We consider the traffic assignment problem. The task is to assign vehicles on a traffic network in order to minimize congestion while satisfying travel demands. Let $\mathcal{A}$, $\mathcal{R}$, and $\mathcal{S}$ be the sets of links, routes, and origin-destination pairs respectively. For every pair $(i,j)\in\mathcal{S}$, let $\mathcal{R}_{i,j}$ and $d_{i,j}$ be the set of routes and the travel demand from $i$ to $j$. Let $x_a$ and $t_a$ be the flow and the travel time on link $a\in\mathcal{A}$, and let $y_r$ be the flow on route $r\in\mathcal{R}$. The Beckmann formulation of the problem \citep{beckmann56}, derived from the Wardrop equilibrium conditions \citep{wardrop52}, is
\begin{alignat}{2}
 \min_{x\in\mathbb{R}^{|\mathcal{A}|}}\;&\sum_{a\in\mathcal{A}}\int_{0}^{x_a}t_a(\xi)\operatorname{d}\!\xi\label{pb:traffic}\\
 \text{s.t.}\;&x_a=\sum_{r\in\mathcal{R}}\mathds{1}_{\{a\in r\}}y_r\quad&&a\in\mathcal{A}\nonumber\\
 &\sum_{r\in\mathcal{R}_{i,j}}y_r=d_{i,j}&&(i,j)\in\mathcal{S}\nonumber\\
 &y_r\geq0&&r\in\mathcal{R}_{i,j},\,(i,j)\in\mathcal{S}.\nonumber
\end{alignat}
A commonly used expression for the travel time $t_a$ as a function of the flow $x_a$, developed by the Bureau of Public Records, is $t_a:x_a\in\mathbb{R}_+\mapsto\tau_a(1+0.15(x_a/c_a)^4)$ where $\tau_a$ and $c_a$ are the free-flow travel time and the capacity of the link. A linear minimization over the feasible region in~\eqref{pb:traffic} amounts to computing the shortest routes between all origin-destination pairs. Thus, the FW oracle is particularly expensive here so we capped the maximum number of rounds in BoostFW to $K\leftarrow5$; see Figure~\ref{fig:angle-improvs-traffic} in Appendix~\ref{apx:angle-improvs}. We implemented the oracle using the function \texttt{all\_pairs\_dijkstra\_path} from the Python package \texttt{networkx} \citep{networkx}. We created a directed acyclic graph with $500$ nodes split into $20$ layers of $25$ nodes each, and randomly dropped links with probability $0.5$ so $|\mathcal{A}|\approx6000$ and $|\mathcal{S}|\approx113000$. We set $d_{i,j}\sim\mathcal{U}(\left[0,1\right])$ for every $(i,j)\in\mathcal{S}$. DICG is not applicable here and AFW-L and BoostFW-L were not competitive. The results are presented in Figure~\ref{fig:traffic}.\\

\begin{figure}[H]
\centering{\includegraphics[scale=0.59]{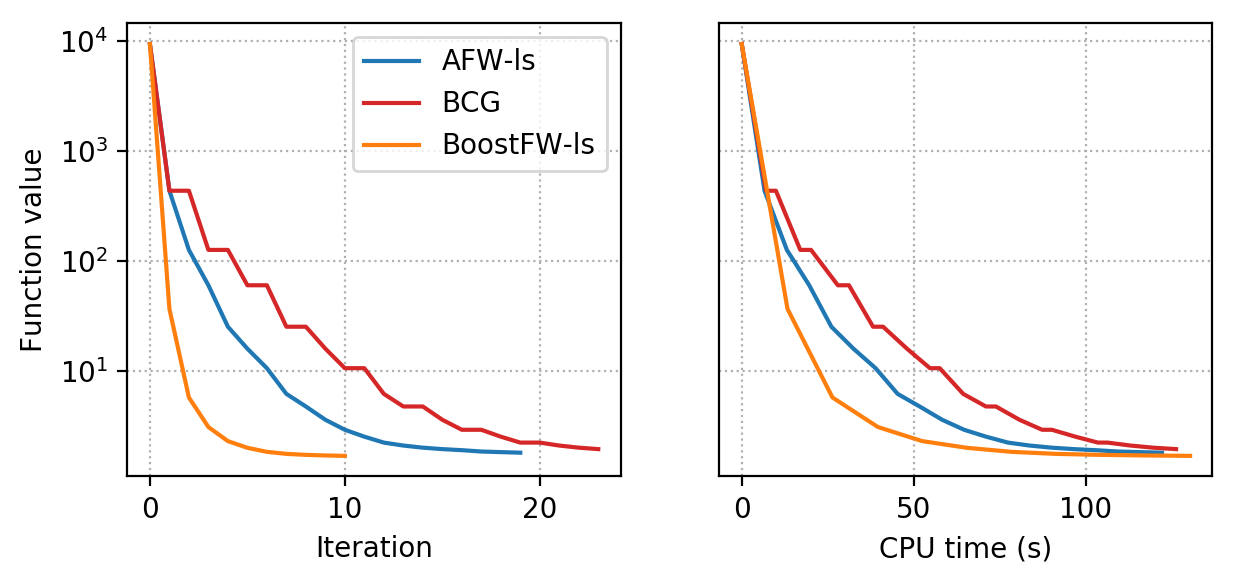}}
\caption{Traffic assignment.}
\label{fig:traffic}
\end{figure}

\subsection{Collaborative filtering}
\label{sec:completion}

We consider the task of collaborative filtering on the MovieLens 100k dataset \citep{harper15movie}, available at \url{https://grouplens.org/datasets/movielens/100k/}. The low-rank assumption on the solution and the approach of \citet{mehta07collabo} lead to the following problem formulation:
\begin{align*}
 \min_{X\in\mathbb{R}^{m\times n}}\;&\frac{1}{|\mathcal{I}|}\sum_{(i,j)\in\mathcal{I}}h_\rho(Y_{i,j}-X_{i,j})\\
 \text{s.t.}\;&\|X\|_{\operatorname{nuc}}\leq\tau
\end{align*}
where $h_\rho$ is the Huber loss with parameter $\rho>0$ \citep{huber64}:
\begin{align*}
 h_\rho:t\in\mathbb{R}\mapsto
 \begin{cases}
  t^2/2&\text{if }|t|\leq\rho\\
  \rho(|t|-\rho/2)&\text{if }|t|>\rho,
 \end{cases}
\end{align*}
$Y\in\mathbb{R}^{m\times n}$ is the given matrix to complete, $\mathcal{I}\subseteq\llbracket1,m\rrbracket\times\llbracket1,n\rrbracket$ is the set of indices of observed entries in $Y$, and $\|\cdot\|_{\operatorname{nuc}}:X\in\mathbb{R}^{m\times n}\mapsto\operatorname{tr}(\sqrt{X^\top X})=\sum_{i=1}^{\min\{m,n\}}\sigma_i(X)$ is the nuclear norm and equals the sum of the singular vectors. It serves as a convex surrogate for the rank constraint \citep{fazel01trace}. Since
\begin{align*}
 \{X\in\mathbb{R}^{m\times n}\mid\|X\|_{\operatorname{nuc}}=1\}
 =\operatorname{conv}(\{uv^\top\mid u\in\mathbb{R}^m,v\in\mathbb{R}^n,\|u\|_2=\|v\|_2=1\}),
\end{align*}
a linear minimization over the nuclear norm-ball of radius $\tau$ amounts to computing the top left and right singular vectors $u$ and $v$ of $-\nabla f(X_t)$ and to return $\tau uv^\top$. To this end, we used the function \texttt{svds} from the Python package \texttt{scipy.sparse.linalg} \citep{scipy}. We have $m=943$, $n=1682$, and $|\mathcal{I}|=10^5$, and we set $\rho=1$, $\tau=5000$, and $L=5\cdot10^{-6}$. DICG is not applicable here. The results are presented in Figure~\ref{fig:completion}.\\

\begin{figure}[H]
\centering{\includegraphics[scale=0.59]{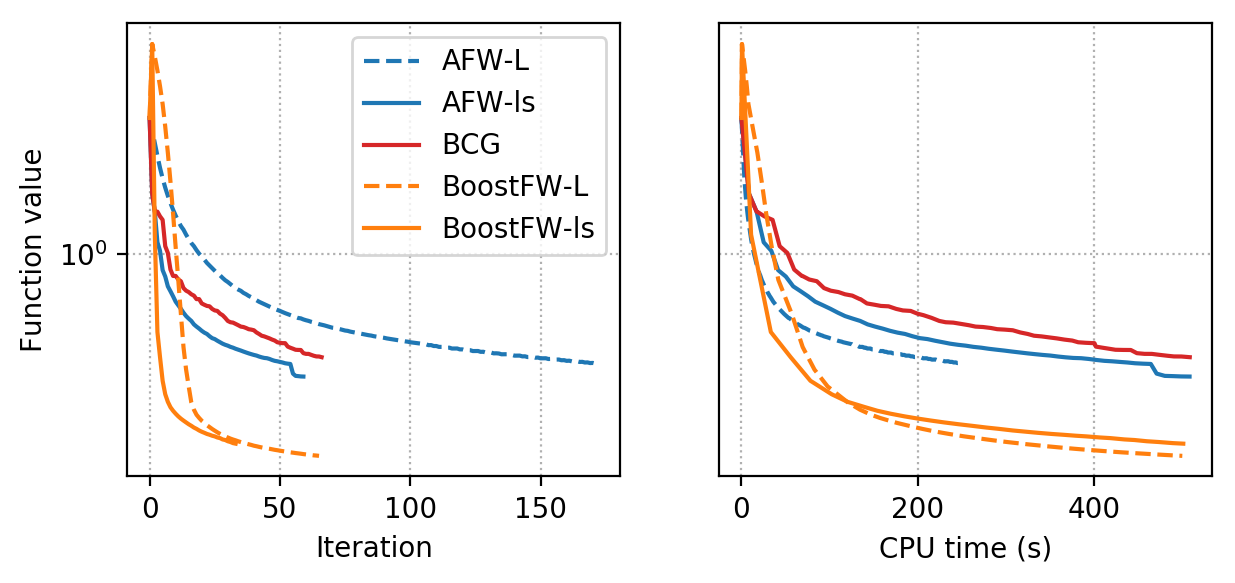}}
\caption{Collaborative filtering on the MovieLens 100k dataset.}
\label{fig:completion}
\end{figure}

The time limit here was set to $500$ seconds but for AFW-L we reduced it to $250$ seconds, else it raises a memory error on our machine shortly after. This is because AFW requires storing the decomposition of the iterate onto $\mathcal{V}$. Note that BoostFW-ls converges faster in CPU time than AFW-L, although it relies on line search, and that BoostFW-L converges faster per iteration than the other methods although it does not rely on line search.

\subsection{Video co-localization}
\label{sec:video}

We consider the task of video co-localization on the aeroplane class of the YouTube-Objects dataset \citep{prest12video}, using the problem formulation of \citet{joulin15video}. The goal is to localize (with bounding boxes) the aeroplane object across the video frames. It consists in minimizing $f:x\in\mathbb{R}^{660}\mapsto x^\top Ax/2+b^\top x$ over a flow polytope, where $A\in\mathbb{R}^{660\times660}$, $b\in\mathbb{R}^{660}$, and the polytope each encode a part of the temporal consistency in the video frames. We obtained the data from \url{https://github.com/Simon-Lacoste-Julien/linearFW}. A linear minimization over the flow polytope amounts to computing a shortest path in the corresponding directed acyclic graph. We implemented the boosting procedure for DICG, which we labeled BoostDICG; see details in Appendix~\ref{apx:boostdicg}. Since the objective function is quadratic, we can derive a closed-form solution to the line search and there is no need for AFW-L or BoostFW-L. We set $\delta\leftarrow10^{-7}$ in BoostFW and $\delta\leftarrow10^{-15}$ in BoostDICG. The results are presented in Figure~\ref{fig:video}.\\

\begin{figure}[h]
\centering{\includegraphics[scale=0.59]{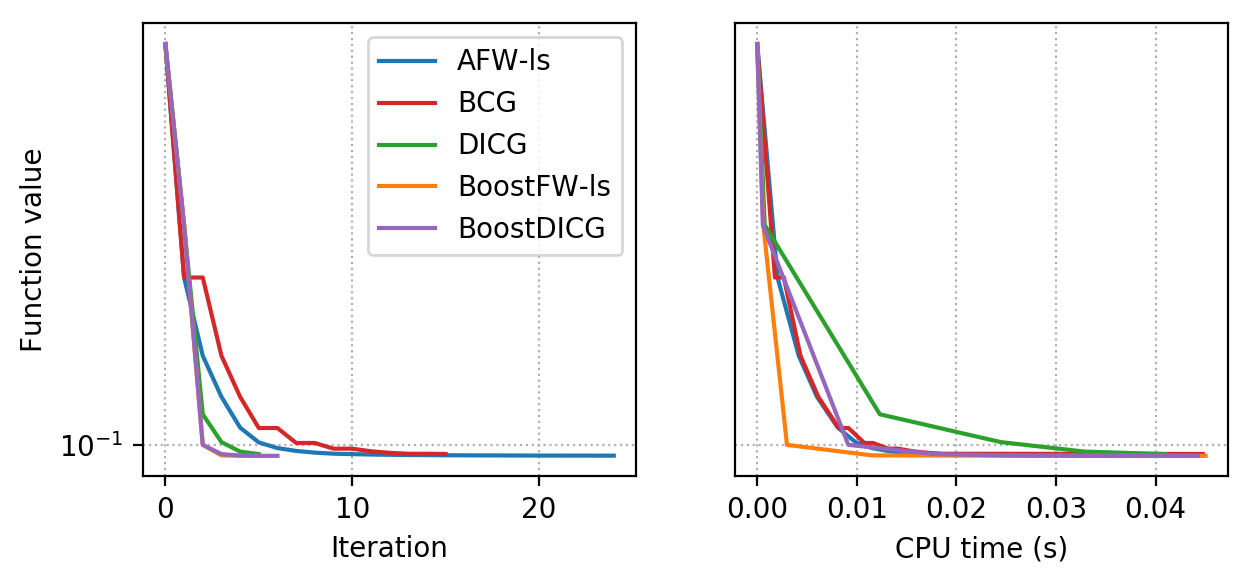}}
\caption{Video co-localization on the YouTube-Objects dataset.}
\label{fig:video}
\end{figure}

All algorithms provide a similar level of performance in function value. In \citet{garber16dicg}, the algorithms are compared with respect to the duality gap $\max_{v\in\mathcal{V}}\langle\nabla f(x_t),x_t-v\rangle$ \citep{jaggi13fw} on the same experiment. For completeness, we report a similar study in Figure~\ref{fig:video-gap}. The boosting procedure applied to DICG produces very promising empirical results.\\

\begin{figure}[h]
\centering{\includegraphics[scale=0.59]{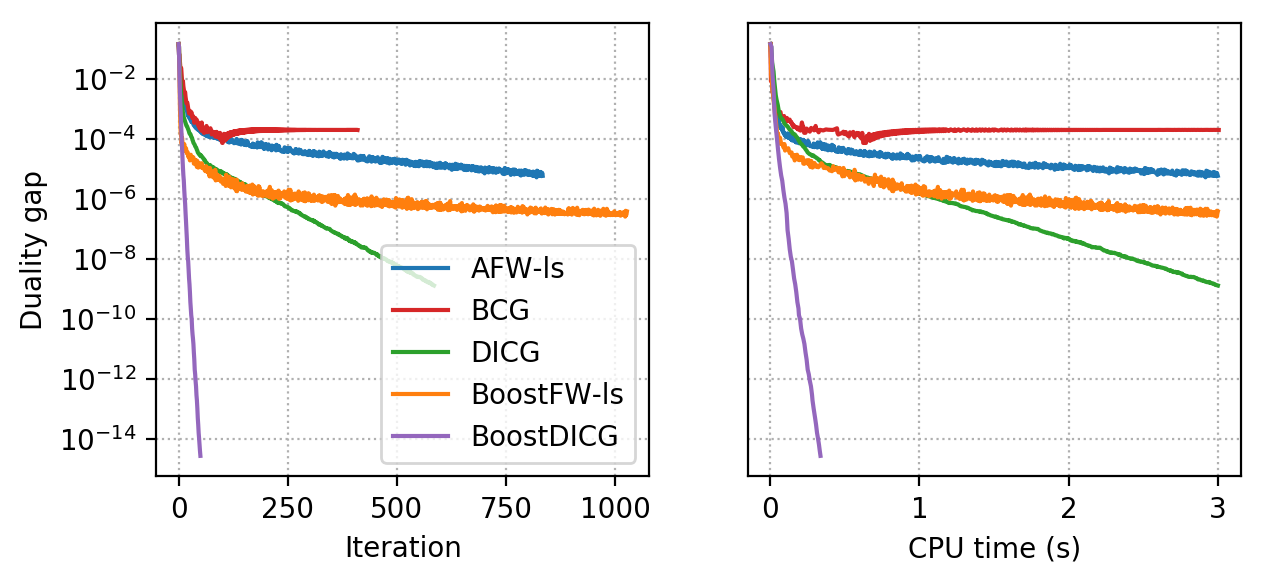}}
\caption{Video co-localization on the YouTube-Objects dataset.}
\label{fig:video-gap}
\end{figure}

Appendix~\ref{apx:full} presents comparisons in duality gap for the other experiments. DICG converges faster than BoostFW in duality gap here (after closing it to $10^{-6}$ though), but it is not the case in the other experiments.

\section{Final remarks}
\label{sec:conclusion}

We have proposed a new and intuitive method to speed up the Frank-Wolfe algorithm by descending in directions better aligned with those of the negative gradients $-\nabla f(x_t)$, all the while remaining projection-free. Our method does not need to maintain the decomposition of the iterates and can naturally be used to boost the performance of any Frank-Wolfe-style algorithm. Although the linear minimization oracle may be called multiple times per iteration, the progress obtained greatly overcomes this cost and leads to strong gains in performance. We demonstrated in a variety of experiments the computational advantage of our method both per iteration and in CPU time over the state-of-the-art. Furthermore, it does not require line search to produce strong performance in practice, which is particularly useful on instances where these are excessively expensive.\\

Future work may replace the gradient pursuit procedure with a faster conic optimization algorithm to potentially reduce the number of oracle calls. It could also be interesting to investigate how to make each oracle call cheaper via, e.g., \emph{lazification} \citep{braun17lazy} or subsampling \citep{kerdreux18subsampling}. Lastly, we expect significant gains in performance when applying our approach to chase the gradient estimators in (non)convex stochastic Frank-Wolfe algorithms as well \citep{svrf16,orgfw20}.

\subsection*{Acknowledgments}

Research reported in this paper was partially supported by NSF CAREER Award CMMI-1452463.

\bibliographystyle{abbrvnat}
{\small\bibliography{biblio}}

\clearpage
\appendix
\onecolumn

\section{Complementary plots}
\label{apx:plots}

\subsection{Lower bound on the number of oracle calls}
\label{apx:lower}

Recall that for any $x\in\mathbb{R}^n$, $\|x\|_0\coloneqq|\{i\in\llbracket1,n\rrbracket\mid[x]_i\neq0\}|$ denotes the number of nonzero entries in $x$. Consider the problem of minimizing $f:x\in\mathbb{R}^n\mapsto\|x\|_2^2$ over the standard simplex $\Delta_n$:
\begin{align*}
 \min_{x\in\Delta_n}\|x\|_2^2.
\end{align*}
Since $\Delta_n$ is the convex hull of the standard basis, Lemma~\ref{lem:lower} establishes a lower bound on the function value of any Frank-Wolfe-style algorithm with respect to the number of oracle calls.

\begin{lemma}[{\citep[Lemma~3]{jaggi13fw}}]
 \label{lem:lower}
 Let $f:x\in\mathbb{R}^n\mapsto\|x\|_2^2$. Then for all $k\in\llbracket1,n\rrbracket$, 
 \begin{align*}
  \min_{\substack{x\in\Delta_n\\\|x\|_0=k}}f(x)=\frac{1}{k}.
 \end{align*}
\end{lemma}

Indeed, suppose that $x_0\in\{e_1,\ldots,e_n\}$ is a standard vector and consider FW (Algorithm~\ref{fw}). FW makes exactly one call to the oracle in each iteration and adds the new vertex to the convex decomposition of the iterate, hence $\|x_t\|_0\leq t+1$ for all $t\in\llbracket1,n-1\rrbracket$. Therefore, Lemma~\ref{lem:lower} shows that the iterates of FW satisfy the lower bound
\begin{align*}
 f(x_t)\geq\frac{1}{t+1}.
\end{align*}
This derivation can be extended to any Frank-Wolfe-style algorithm by comparing the function value at iteration $t$ vs.~the number of oracle calls up to iteration $t$. In Figure~\ref{fig:lower}, we demonstrate that although BoostFW may call the oracle multiple times per iteration, it is still compatible with the lower bound. We set $n=1000$ and since the objective is quadratic, we used an exact line search step-size strategy in FW-ls and BoostFW-ls. Note that the optimal value of the problem is $1/n=10^{-3}$.\\

\begin{figure}[H]
\centering{\includegraphics[scale=0.59]{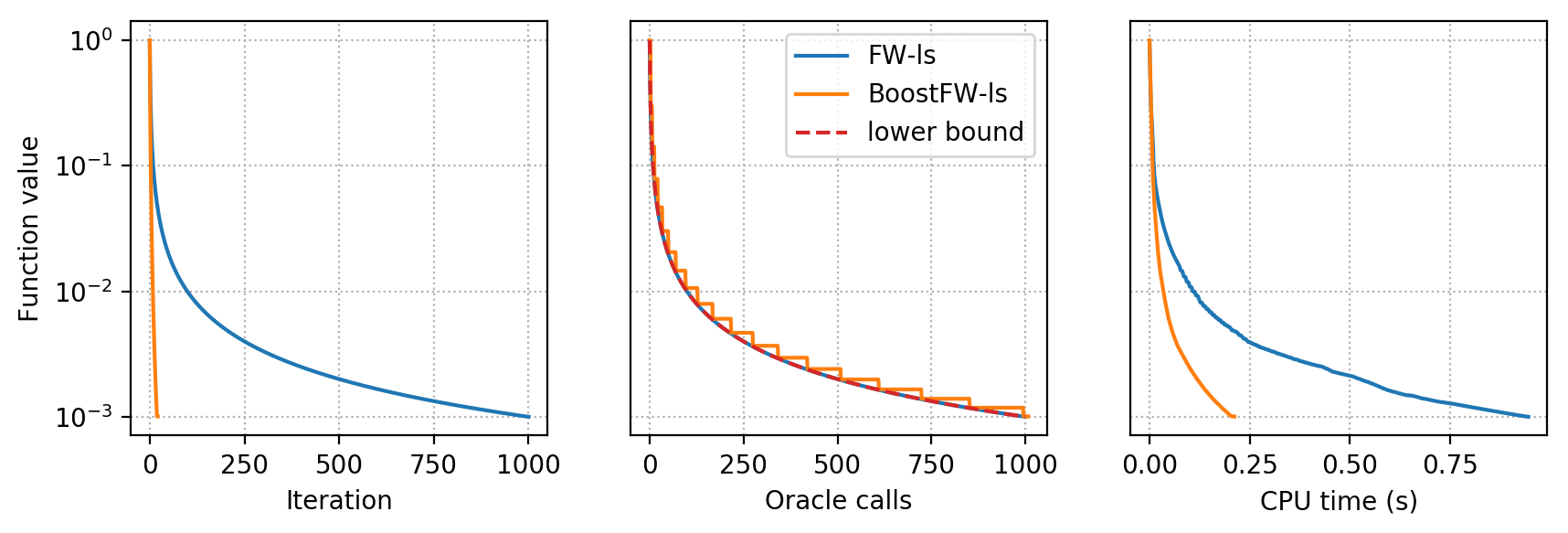}}
\caption{Lower bound on the number of oracle calls.}
\label{fig:lower}
\end{figure}

\clearpage
\subsection{Illustration of the improvements in alignment during the gradient pursuit procedure}
\label{apx:angle-improvs}

We define the relative improvement in alignment between rounds $k-1$ and $k\in\llbracket2,K_t\rrbracket$ in the gradient pursuit procedure at iteration $t\in\llbracket0,T-1\rrbracket$ of BoostFW (Algorithm~\ref{boofw}) as
\begin{align*}
 \theta_{t,k}
 \coloneqq\frac{\operatorname{align}(-\nabla f(x_t),d_k)-\operatorname{align}(-\nabla f(x_t),d_{k-1})}{\operatorname{align}(-\nabla f(x_t),d_{k-1})}.
\end{align*}
For a fixed round $k$, we plot in Figure~\ref{fig:angle-improvs} the mean of $\theta_{t,k}$ across all iterations $t$ that performed a $k$-th round, i.e., 
\begin{align*}
 \theta_k\coloneqq\frac{1}{|\{t\in\llbracket0,T-1\rrbracket\mid k\leq K_t\}|}\sum_{t=0}^{T-1}\theta_{t,k}\mathds{1}_{\{k\leq K_t\}},
\end{align*}
in the sparse signal recovery experiment (Section~\ref{sec:signal}). The error bars represent $\pm1$ standard deviation. We see that on average the second round produces an improvement in alignment of $\sim32\%$, the third round produces an improvement of $\sim16\%$, etc. In particular, the plot could suggest that $7$ rounds in each iteration are enough.

\begin{figure}[H]
\centering{\includegraphics[scale=0.59]{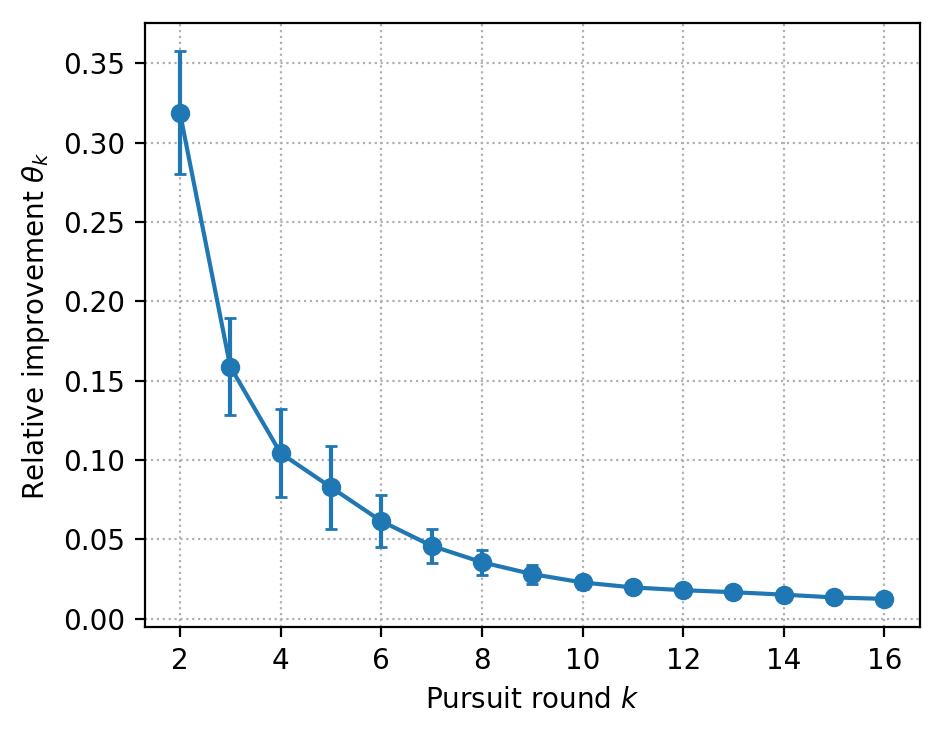}}
\caption{Relative improvements in alignment during the gradient pursuit procedure of BoostFW.}
\label{fig:angle-improvs}
\end{figure}

In the traffic assignment experiment (Section~\ref{sec:traffic}), the FW oracle is particularly expensive so we decided to cap the maximum number of rounds $K$. We plot in Figure~\ref{fig:angle-improvs-traffic} the relative improvements in alignment, and we chose to set $K\leftarrow5$.

\begin{figure}[H]
\centering{\includegraphics[scale=0.59]{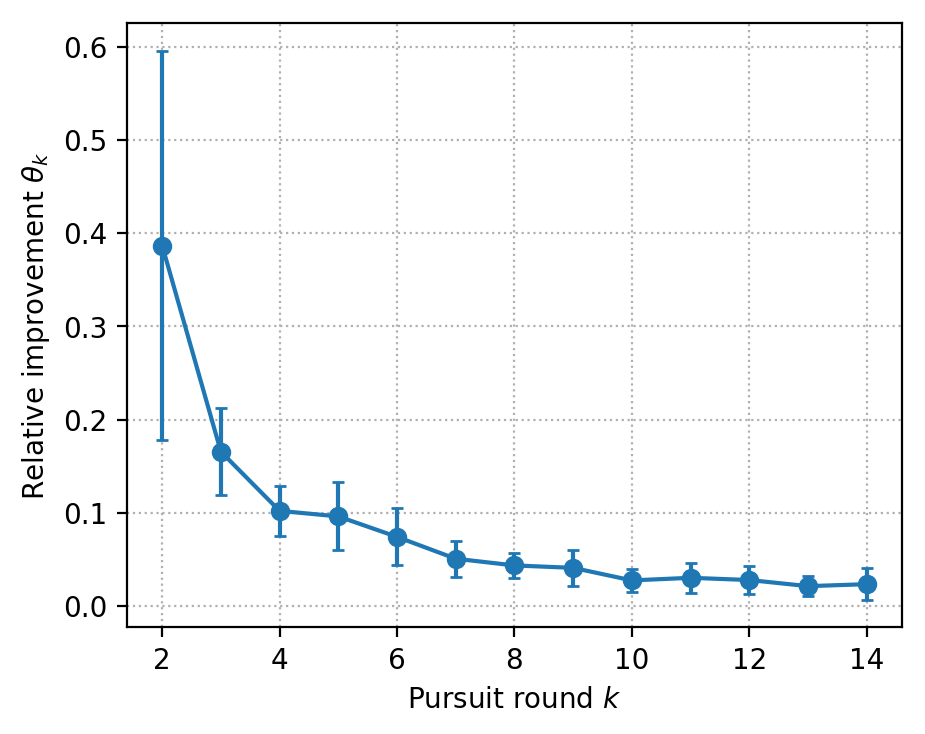}}
\caption{Traffic assignment (Section~\ref{sec:traffic}).}
\label{fig:angle-improvs-traffic}
\end{figure}

\clearpage
\subsection{Computational experiments}
\label{apx:full}

Here we provide additional plots for each experiment of Section~\ref{sec:exp}: comparisons in number of oracle calls and in duality gap. The duality gap is $\max_{v\in\mathcal{V}}\langle\nabla f(x_t),x_t-v\rangle$ \citep{jaggi13fw} and we did not account for the CPU time taken to plot it. In number of oracle calls, the plots have a stair-like behavior as multiple calls can be made within an iteration. We see that BoostFW performs more oracle calls than the other methods in general however it converges faster both per iteration and in CPU time. Note that in the traffic assignment experiment (Figure~\ref{fig:traffic-oracles}), BoostFW also converges faster per oracle call. In the sparse logistic regression experiment (Figure~\ref{fig:gisette-oracles}), the line search-free strategies converge faster in CPU time and the line search strategies converge faster per iteration, but in the collaborative filtering experiment (Figure~\ref{fig:huber-oracles}), BoostFW-ls and BoostFW-L respectively converge faster than expected.\\

\begin{figure}[H]
\centering{\includegraphics[scale=0.59]{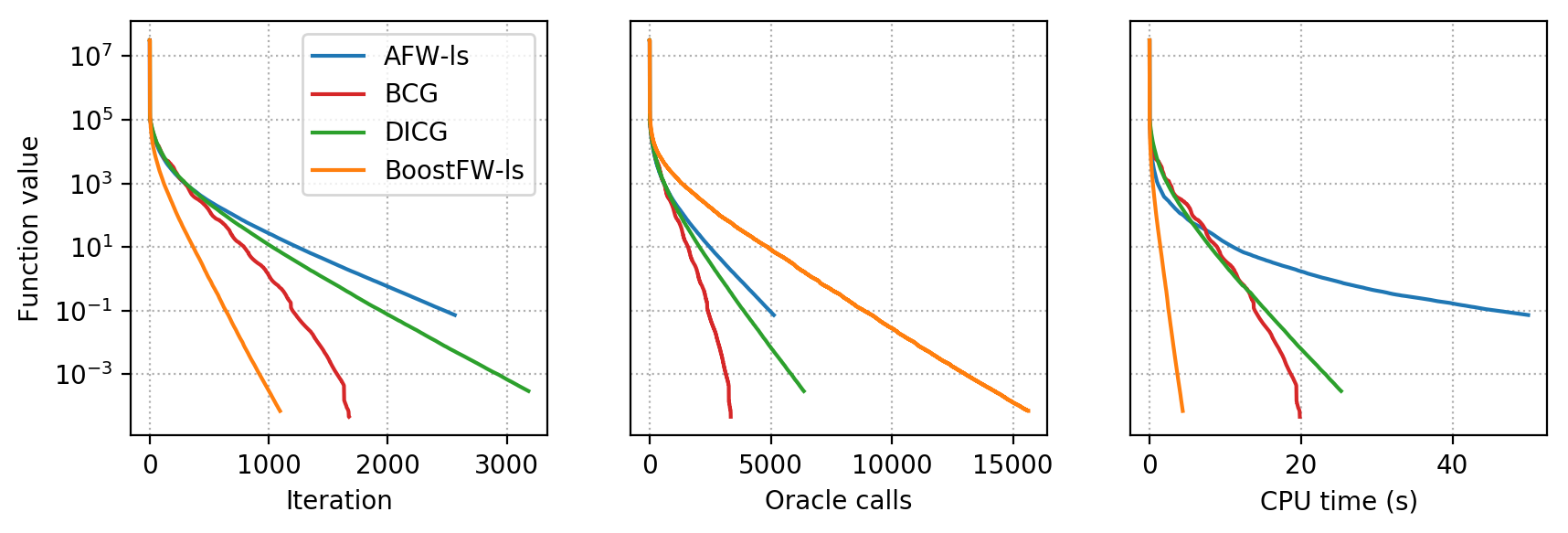}}
\centering{\includegraphics[scale=0.59]{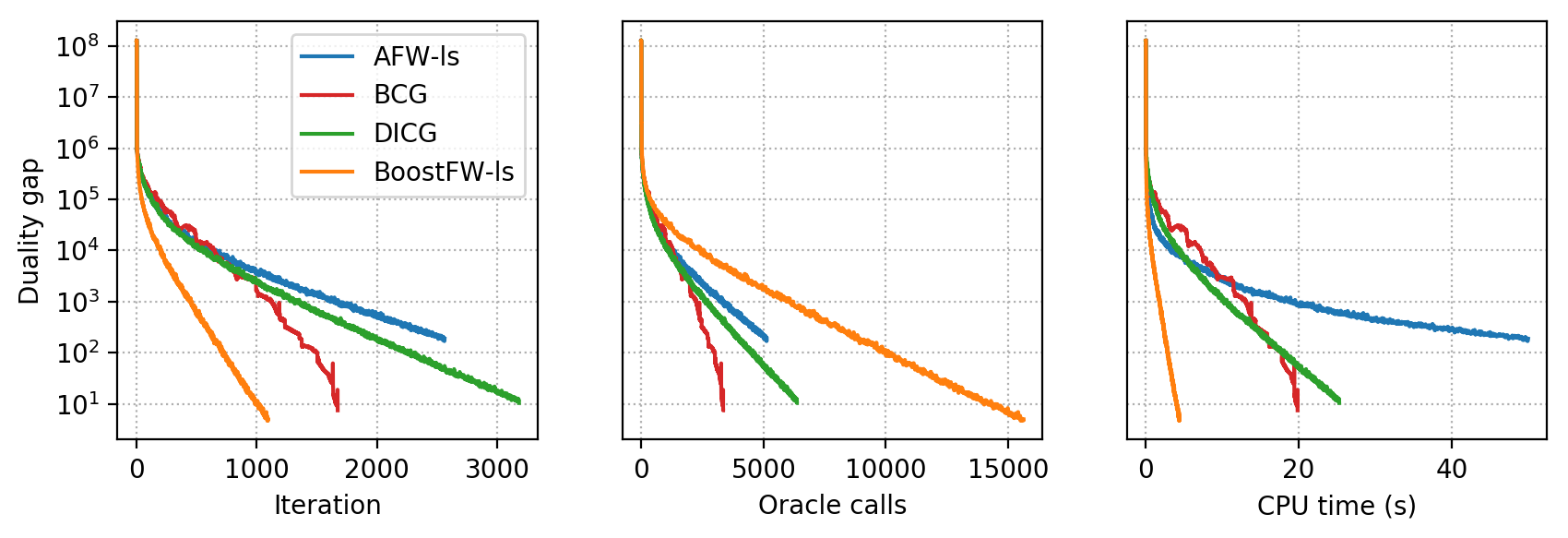}}
\caption{Sparse signal recovery (Section~\ref{sec:signal}).}
\label{fig:signal-oracles}
\end{figure}

\begin{figure}[H]
\centering{\includegraphics[scale=0.59]{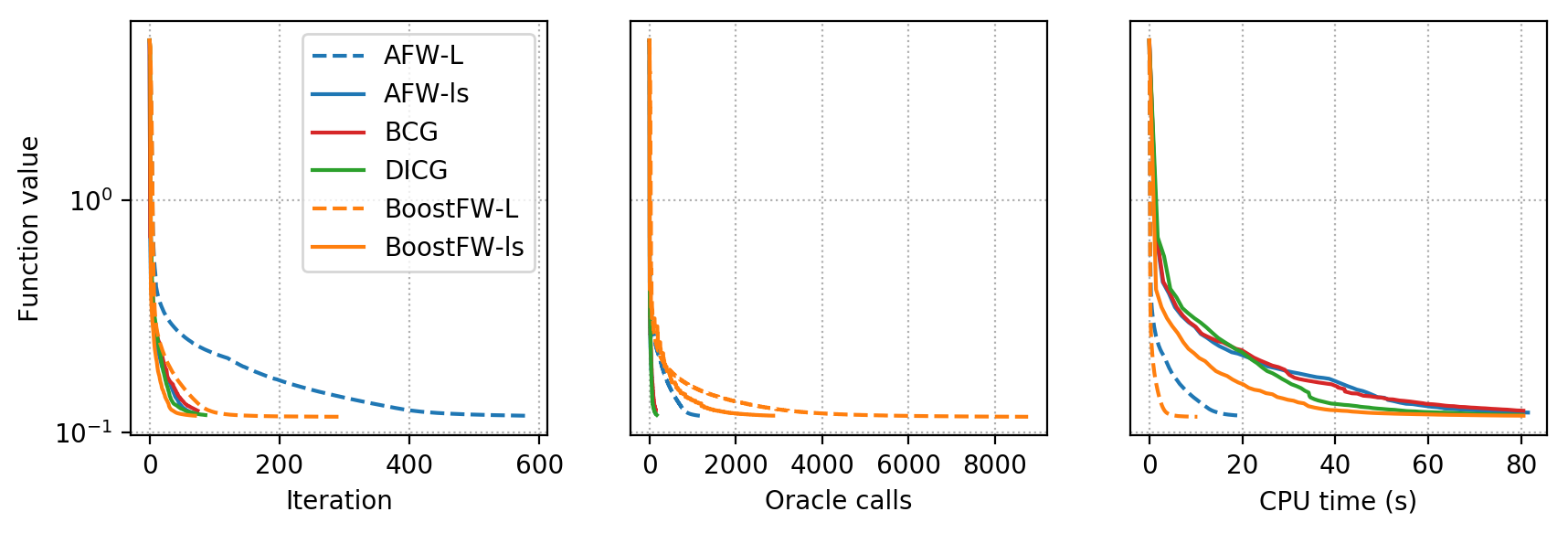}}
\centering{\includegraphics[scale=0.59]{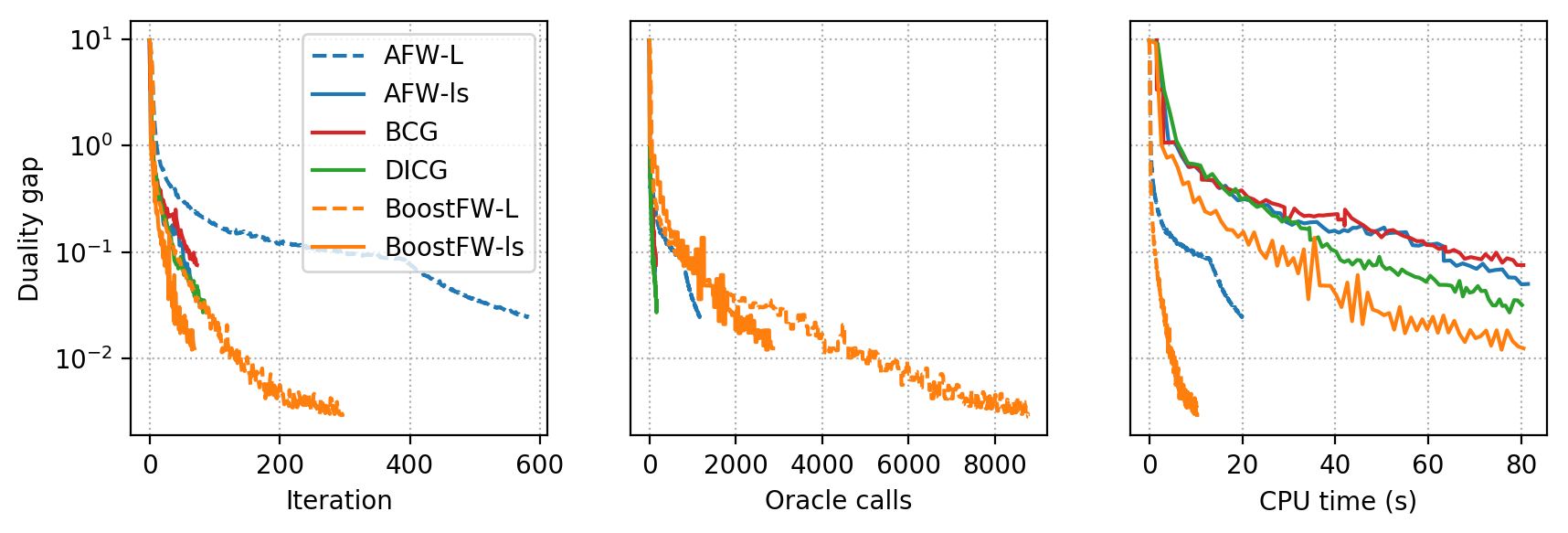}}
\caption{Sparse logistic regression on the Gisette dataset (Section~\ref{sec:gisette}).}
\label{fig:gisette-oracles}
\end{figure}

\vfill

\begin{figure}[H]
\centering{\includegraphics[scale=0.59]{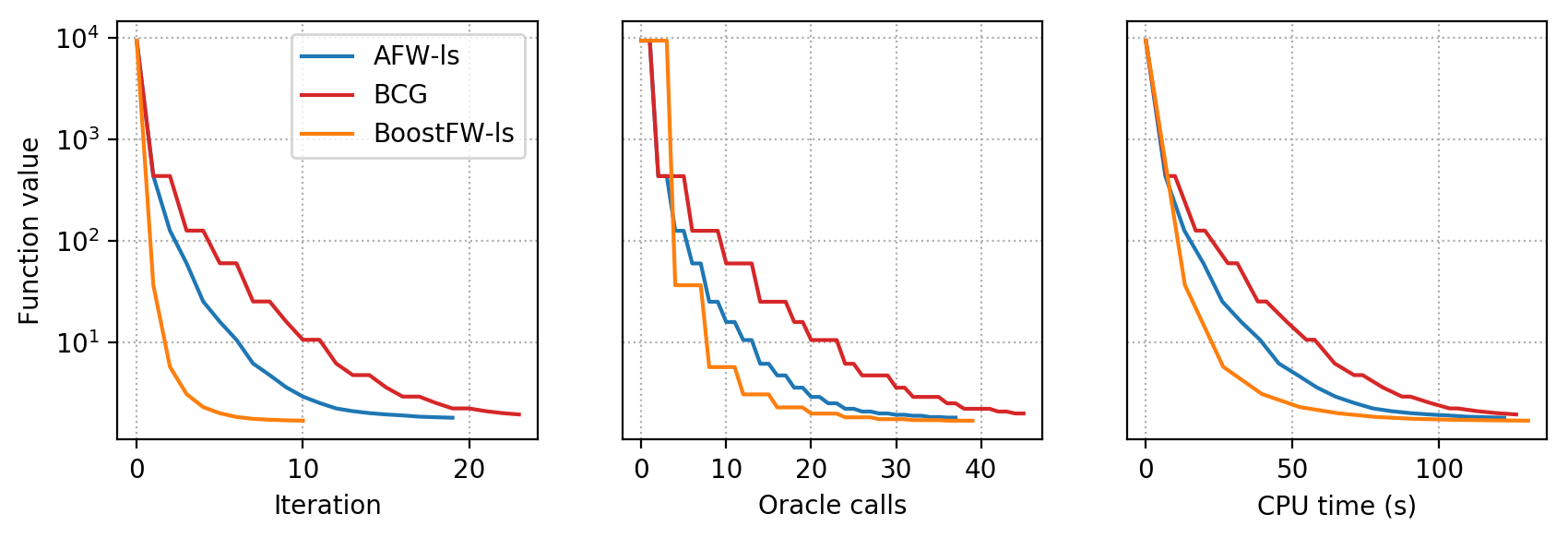}}
\centering{\includegraphics[scale=0.59]{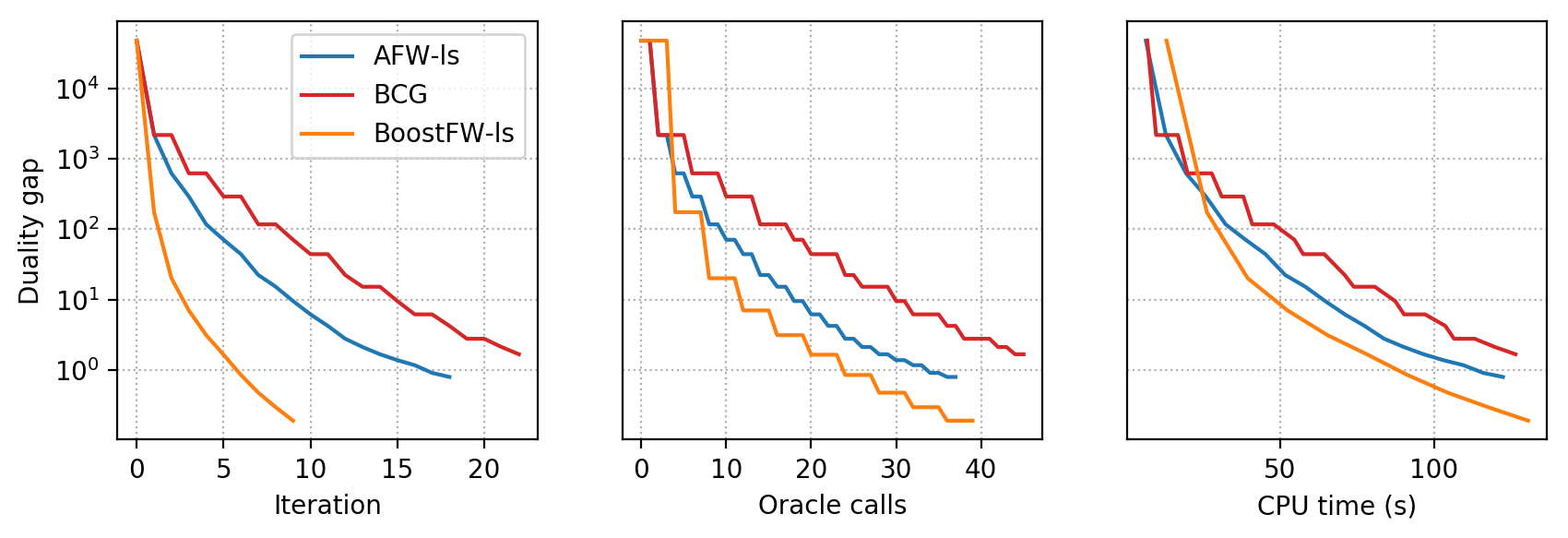}}
\caption{Traffic assignment (Section~\ref{sec:traffic}).}
\label{fig:traffic-oracles}
\end{figure}

\begin{figure}[H]
\centering{\includegraphics[scale=0.59]{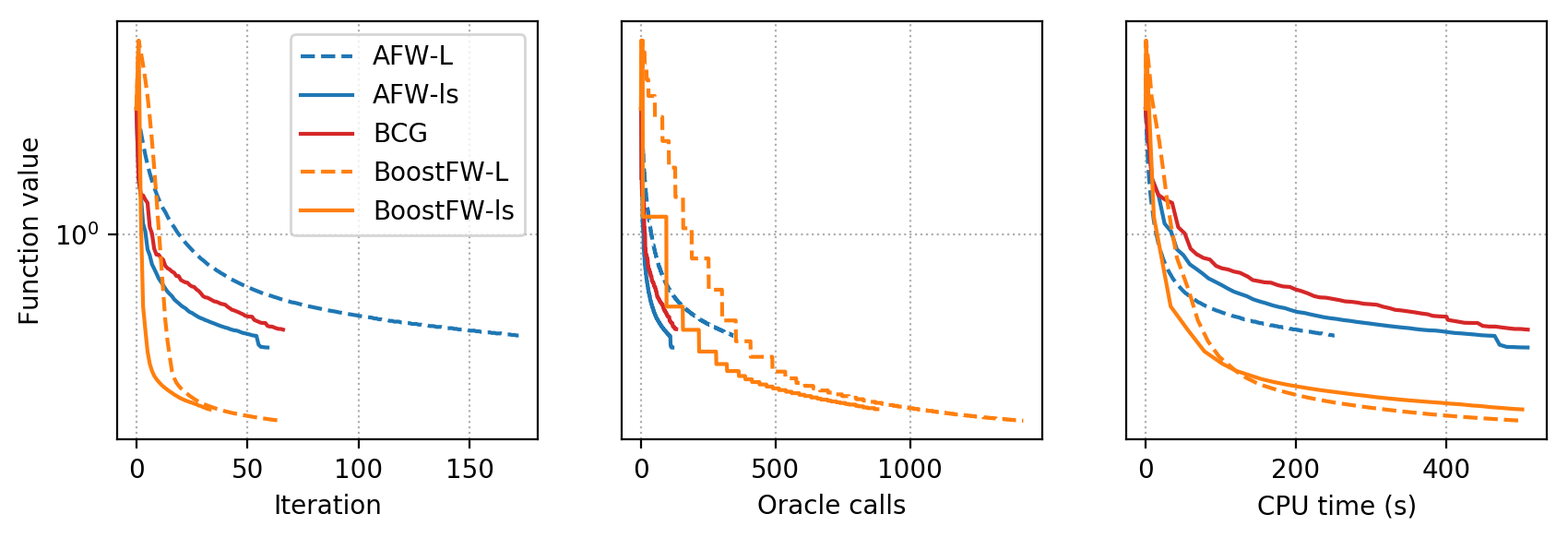}}
\centering{\includegraphics[scale=0.59]{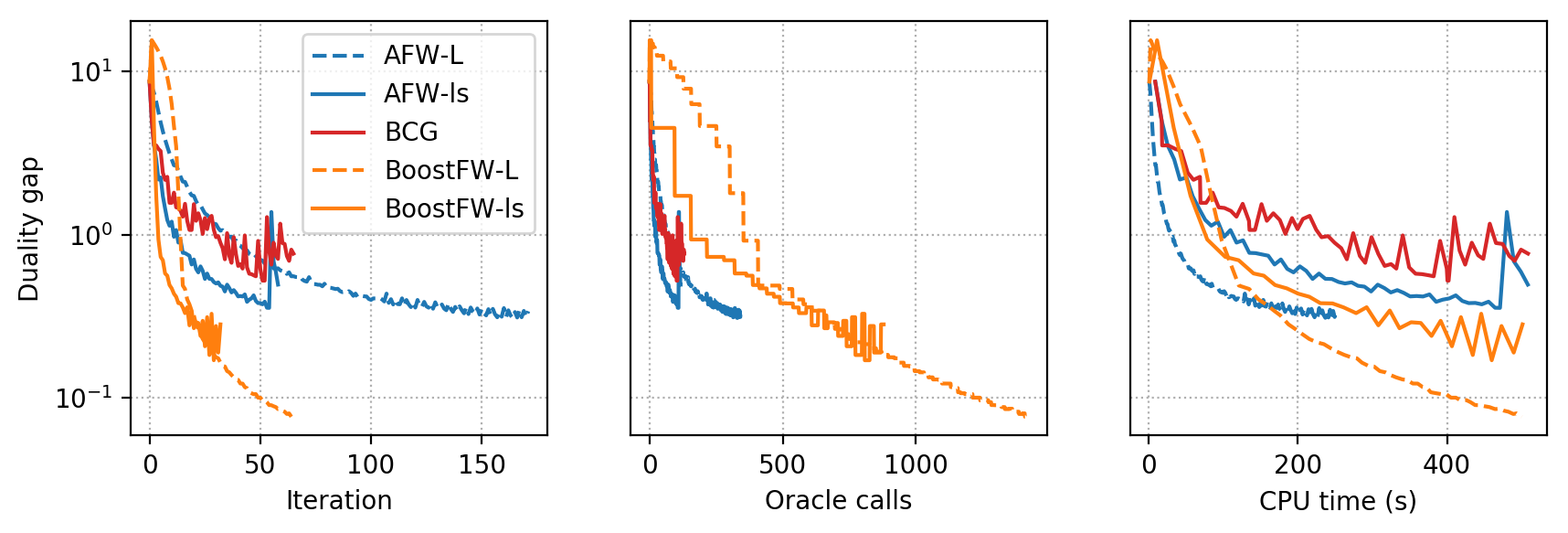}}
\caption{Collaborative filtering on the MovieLens 100k dataset (Section~\ref{sec:completion}).}
\label{fig:huber-oracles}
\end{figure}

\vfill

\begin{figure}[H]
\centering{\includegraphics[scale=0.59]{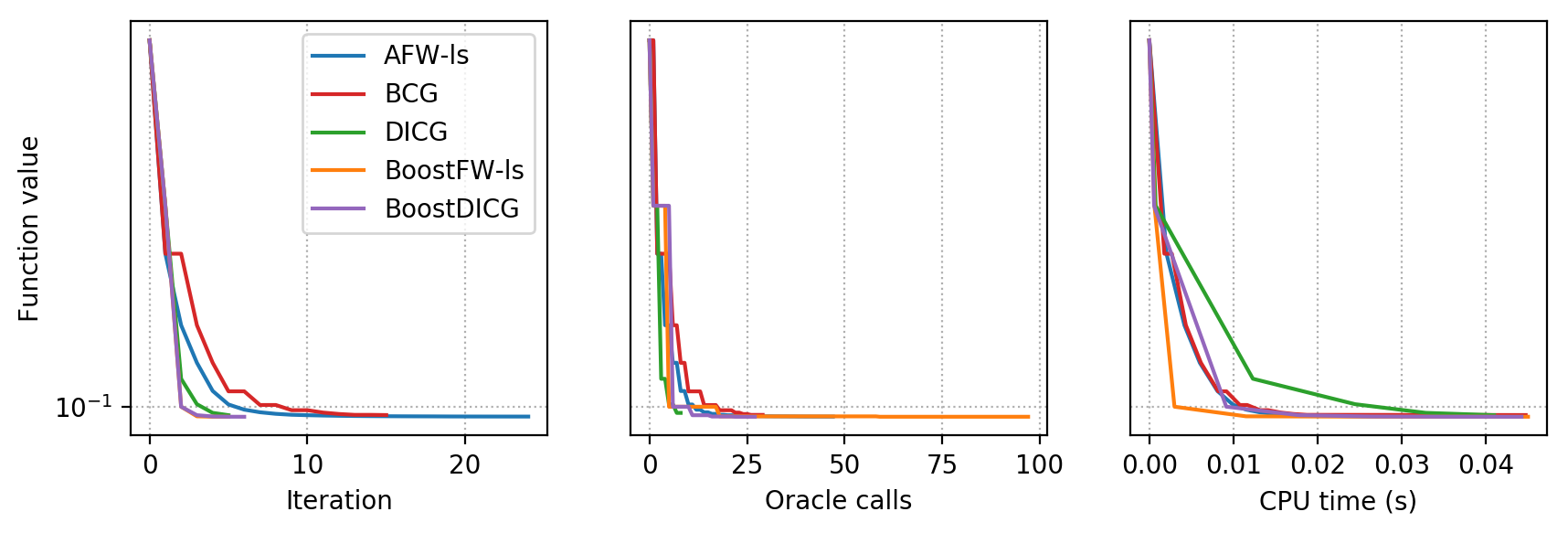}}
\centering{\includegraphics[scale=0.59]{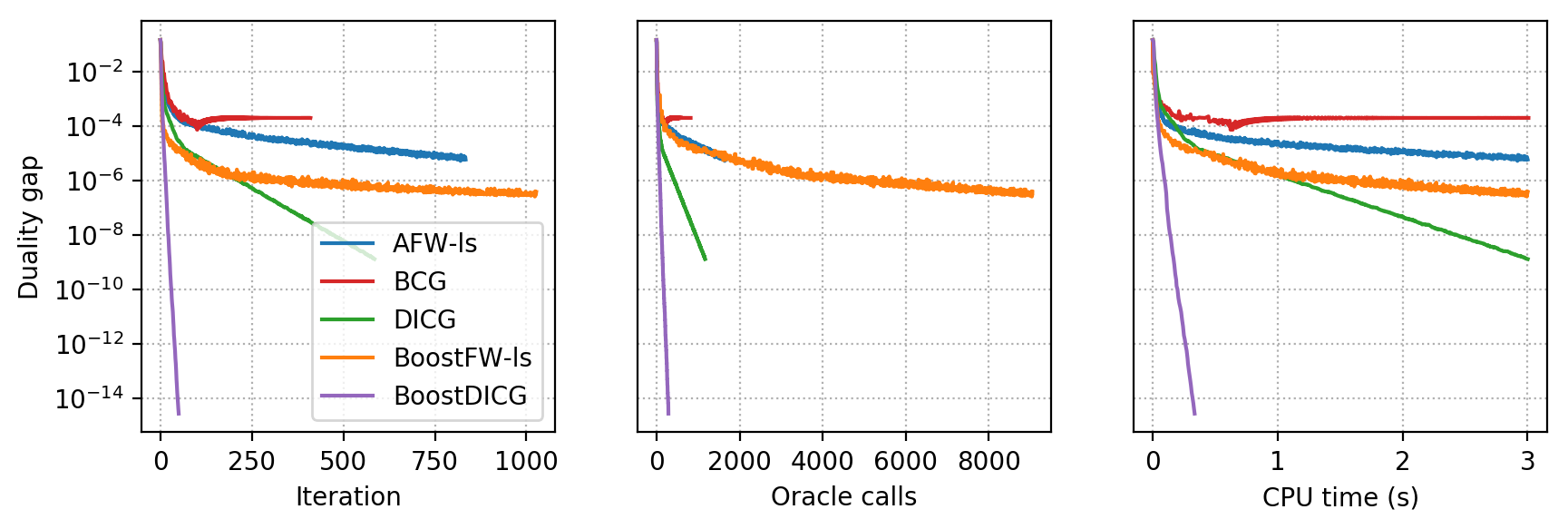}}
\caption{Video co-localization on the YouTube-Objects dataset (Section~\ref{sec:video}).}
\label{fig:video-oracles}
\end{figure}

\clearpage
\section{Boosting DICG}
\label{apx:boostdicg}

We present an application of the boosting procedure to another Frank-Wolfe-style algorithm. Although the Away-Step and Pairwise Frank-Wolfe algorithms \citep{lacoste15linear} are more similar in essence to the vanilla Frank-Wolfe algorithm, we chose to apply our approach to the Decomposition-Invariant Pairwise Conditional Gradient (DICG) \citep{garber16dicg} because it does not need to maintain the decomposition of the iterates, which is a very favorable property in practice.\\

We recall DICG in Algorithm~\ref{dicg} and present our proposition for BoostDICG in Algorithm~\ref{boostdicg}. Notice that since DICG moves in the pairwise direction $v_t^\text{FW}-v_t^\text{away}$, in BoostDICG we chase the direction of $-\nabla f(x_t)$ from $v_t^\text{away}$ (and not from $x_t$). Proposition~\ref{prop:boostdicg} shows that the iterates of BoostDICG are feasible. Similarly to DICG, BoostDICG is applicable only to polytopes of the form $\mathcal{P}=\{x\in\mathbb{R}^n\mid Ax=b,x\geq0\}$ with set of vertices $\mathcal{V}\subseteq\{0,1\}^n$, and it does not need to maintain the decomposition of the iterates. See also \citet{bashiri17dicg} for a follow-up work extending DICG to arbitrary polytopes.\\

\begin{algorithm}[h]
\caption{Decomposition-Invariant Pairwise Conditional Gradient (DICG)}
\label{dicg}
\textbf{Input:} Start point $x_0\in\mathcal{P}$.\\
\textbf{Output:} Point $x_T\in\mathcal{P}$.
\begin{algorithmic}[1]
\STATE$x_1\leftarrow\argmin\limits_{v\in\mathcal{V}}\langle\nabla f(x_0),v\rangle$
\FOR{$t=1$ \textbf{to} $T-1$}
\STATE$v_t^\text{FW}\leftarrow\argmin\limits_{v\in\mathcal{V}}\langle\nabla f(x_t),v\rangle$\hfill$\triangleright${ FW oracle}
\STATE$\left[\tilde{\nabla}f(x_t)\right]_i\leftarrow
\begin{cases}
 [\nabla f(x_t)]_i&\textbf{if }[x_t]_i>0\\
 -\infty&\textbf{if }[x_t]_i=0
\end{cases}
\quad\textbf{ for }i\in\llbracket1,n\rrbracket$
\STATE$v_t^\text{away}\leftarrow\argmax\limits_{v\in\mathcal{V}}\langle\tilde{\nabla}f(x_t),v\rangle$\hfill$\triangleright${ FW oracle}
\STATE$\bar{\gamma}_t\leftarrow\max\{\gamma\in\left[0,1\right]\mid x_t+\gamma(v_t^\text{FW}-v_t^\text{away})\geq0\}$
\STATE$\gamma_t\leftarrow\argmin\limits_{\gamma\in\left[0,\bar{\gamma}_t\right]}f(x_t+\gamma(v_t^\text{FW}-v_t^\text{away}))$
\STATE$x_{t+1}\leftarrow x_t+\gamma_t(v_t^\text{FW}-v_t^\text{away})$
\ENDFOR
\end{algorithmic}
\end{algorithm}

\begin{algorithm}[h]
\caption{Boosted Decomposition-Invariant Pairwise Conditional Gradient (BoostDICG)}
\label{boostdicg}
\textbf{Input:} Input point $y\in\mathcal{C}$, maximum number of rounds $K\in\mathbb{N}\backslash\{0\}$, alignment improvement tolerance $\delta\in\left]0,1\right[$, step-size strategy $\gamma_t\in\left[0,1\right]$.\\
\textbf{Output:} Point $x_T\in\mathcal{P}$.
\begin{algorithmic}[1]
\STATE$x_0\leftarrow\argmin\limits_{v\in\mathcal{V}}\langle\nabla f(y),v\rangle$
\FOR{$t=0$ \textbf{to} $T-1$}
\STATE$\left[\tilde{\nabla}f(x_t)\right]_i\leftarrow
\begin{cases}
 [\nabla f(x_t)]_i&\textbf{if }[x_t]_i>0\\
 -\infty&\textbf{if }[x_t]_i=0
\end{cases}
\quad$ \textbf{for} $i\in\llbracket1,n\rrbracket$
\STATE$v_t^\text{away}\leftarrow\argmax\limits_{v\in\mathcal{V}}\langle\tilde{\nabla}f(x_t),v\rangle$\hfill$\triangleright${ FW oracle}

\STATE$d_0\leftarrow0$
\STATE$\Lambda_t\leftarrow0$
\FOR{$k=0$ \textbf{to} $K-1$}
\STATE$r_k\leftarrow-\nabla f(x_t)-d_k$\hfill$\triangleright${ $k$-th residual}
\STATE$v_k\leftarrow\argmax\limits_{v\in\mathcal{V}}\langle r_k,v\rangle$\hfill$\triangleright${ FW oracle}
\STATE$u_k\leftarrow\argmax\limits_{u\in\{v_k-v_t^\text{away},-d_k/\|d_k\|\}}\langle r_k,u\rangle$
\STATE$\lambda_k\leftarrow\displaystyle\frac{\langle r_k,u_k\rangle}{\|u_k\|^2}$
\STATE$d_{k+1}'\leftarrow d_k+\lambda_ku_k$
\IF{$\operatorname{align}(-\nabla f(x_t),d_{k+1}')-\operatorname{align}(-\nabla f(x_t),d_k)\geq\delta$}
\STATE$d_{k+1}\leftarrow d_{k+1}'$
\STATE$\Lambda_t\leftarrow\begin{cases}\Lambda_t+\lambda_k&\textbf{if }u_k=v_k-v_t^\text{away}\\\Lambda_t(1-\lambda_k/\|d_k\|)&\textbf{if }u_k=-d_k/\|d_k\|\end{cases}$
\ELSE
\STATE\textbf{break}\hfill$\triangleright${ exit $k$-loop}
\ENDIF
\ENDFOR
\STATE$K_t\leftarrow k$
\STATE$g_t\leftarrow d_{K_t}/\Lambda_t$\hfill$\triangleright${ normalization}

\STATE$\bar{\gamma}_t\leftarrow\max\{\gamma\in\left[0,1\right]\mid x_t+\gamma g_t\geq0\}$
\STATE$\gamma_t\leftarrow\argmin\limits_{\gamma\in\left[0,\bar{\gamma}_t\right]}f(x_t+\gamma g_t)$
\STATE$x_{t+1}\leftarrow x_t+\gamma_tg_t$
\ENDFOR
\end{algorithmic}
\end{algorithm}

\begin{proposition}
\label{prop:boostdicg}
 The iterates of BoostDICG (Algorithm~\ref{boostdicg}) are feasible.
\end{proposition}

\begin{proof}
 We proceed by induction. By definition, $x_0\in\argmin_{v\in\mathcal{V}}\langle\nabla f(y),v\rangle\subseteq\mathcal{P}$. Let $t\in\llbracket0,T-1\rrbracket$ and suppose that $x_t\in\mathcal{P}$. Then $x_{t+1}\leftarrow x_t+\gamma_tg_t$ where $\gamma_t\in\left[0,\bar{\gamma}_t\right]$ and $\bar{\gamma}_t\leftarrow\max\{\gamma\in\left[0,1\right]\mid x_t+\gamma g_t\geq0\}$. Similarly to the proof of Proposition~\ref{prop}\ref{prop:gx}, we can show that $g_t+v_t^\text{away}\in\mathcal{P}$. Thus, since $x_{t+1}=x_t+\gamma_t((g_t+v_t^\text{away})-v_t^\text{away})$ and $Ax_t=b$, $A(g_t+v_t^\text{away})=b$, and $Av_t^\text{away}=b$ by feasibility, we have $Ax_{t+1}=b$. Lastly, $x_t+\bar{\gamma}_tg_t\geq0$ so, since $x_t\geq0$ and $\gamma_t\in\left[0,\bar{\gamma}_t\right]$, we have $x_t+\gamma_tg_t\geq0$, i.e., $x_{t+1}\geq0$. Therefore, $x_{t+1}\in\mathcal{P}$. 
\end{proof}

\clearpage
\section{A result on the Away-Step Frank-Wolfe algorithm}
\label{apx:afw}

We first recall the Away-Step Frank-Wolfe algorithm (AFW) \citep{wolfe70} in Algorithm~\ref{afw} and its convergence rate over polytopes with line search in Theorem~\ref{th:afwls} \citep{lacoste15linear}. This analysis is based on the pyramidal width of the polytope and the geometric strong convexity of the objective function (Lemma~\ref{lem:geom}). We refer the reader to \citet[Section~3]{lacoste15linear} for the definition of the pyramidal width.\\

Let $\mathcal{V}\subset\mathcal{P}$ denote the set of vertices of the polytope $\mathcal{P}$. In Algorithm~\ref{afw}, $\lambda_t\in\mathbb{R}^{|\mathcal{V}|}$ denotes the distribution of coefficients of the convex decomposition of $x_t$ over $\mathcal{V}$ and $\lambda_t(v)$ is the coefficient of vertex $v\in\mathcal{V}$ (as determined by the algorithm). Note that when an away step is taken and $\gamma_t=\gamma_{\max}$, then $\lambda_{t+1}(v_t^\text{away})=0$ where $v_t^\text{away}\in\mathcal{S}_t$. Thus, these steps always decrease the size of the active set: $|\mathcal{S}_{t+1}|<|\mathcal{S}_t|$. They are often referred to as \emph{drop} steps.\\

\begin{algorithm}[h]
\caption{Away-Step Frank-Wolfe (AFW)}
\label{afw}
\textbf{Input:} Start vertex $x_0\in\mathcal{V}$, step-size strategy $\gamma_t\in\left[0,1\right]$.\\
\textbf{Output:} Point $x_T\in\mathcal{P}$.
\begin{algorithmic}[1]
\STATE$\mathcal{S}_0\leftarrow\{x_0\}$\hfill$\triangleright${ active set}
\STATE$\lambda_0\leftarrow\{\mathds{1}_{\{v=x_0\}}
\textbf{ for }v\in\mathcal{V}\}$\hfill$\triangleright${ distribution of coefficients}
\FOR{$t=0$ \textbf{to} $T-1$}
\STATE$v_t^\text{FW}\leftarrow\argmin\limits_{v\in\mathcal{V}}\langle\nabla f(x_t),v\rangle$\hfill$\triangleright${ FW oracle}\label{afw:vfw}
\STATE$v_t^\text{away}\leftarrow\argmax\limits_{v\in\mathcal{S}_t}\langle\nabla f(x_t),v\rangle$
\IF{$\langle\nabla f(x_t),x_t-v_t^\text{FW}\rangle\geq\langle\nabla f(x_t),v_t^\text{away}-x_t\rangle$}\label{afw:criterion}
\STATE$x_{t+1}\leftarrow x_t+\gamma_t(v_t^\text{FW}-x_t)$\hfill$\triangleright${ FW step}\label{afw:fw}
\STATE$\mathcal{S}_{t+1}\leftarrow\mathcal{S}_t\cup\{v_t^\text{FW}\}$
\STATE$\lambda_{t+1}(v)\leftarrow(1-\gamma_t)\lambda_t(v)+\gamma_t\mathds{1}_{\{v=v_t^\text{FW}\}}\textbf{ for }v\in\mathcal{S}_{t+1}$
\ELSE\label{afw:criterion2}
\STATE$\gamma_{\max}\leftarrow\lambda_t(v_t^\text{away})/(1-\lambda_t(v_t^\text{away}))$
\STATE$x_{t+1}\leftarrow x_t+\gamma_t(x_t-v_t^\text{away})$\hfill$\triangleright${ away step}\label{afw:away}
\STATE$\mathcal{S}_{t+1}\leftarrow\mathcal{S}_t$
\STATE$\lambda_{t+1}(v)\leftarrow(1+\gamma_t)\lambda_t(v)-\gamma_t\mathds{1}_{\{v=v_t^\text{away}\}}\textbf{ for }v\in\mathcal{S}_{t+1}$
\ENDIF
\STATE$\mathcal{S}_{t+1}\leftarrow\{v\in\mathcal{S}_{t+1}\mid\lambda_{t+1}(v)>0\}$
\ENDFOR
\end{algorithmic}
\end{algorithm}

\begin{lemma}[{\citep[Equations~(28) and~(23)]{lacoste15linear}}]
 \label{lem:geom}
 Let $\mathcal{P}\subset\mathbb{R}^n$ be a polytope with pyramidal width $W>0$ and $f:\mathbb{R}^n\rightarrow\mathbb{R}$ be a $S$-strongly convex function. Then the Away-Step Frank-Wolfe algorithm (AFW, Algorithm~\ref{afw}) ensures for all $t\in\llbracket0,T-1\rrbracket$,
 \begin{align*}
  f(x_t)-\min_\mathcal{P}f
  \leq\frac{\langle\nabla f(x_t),v_t^\text{away}-v_t^\text{FW}\rangle^2}{2SW^2}.
 \end{align*}
\end{lemma}

\begin{theorem}[(AFW with line search, {\citet[Theorem~1]{lacoste15linear}})]
 \label{th:afwls}
 Let $\mathcal{P}\subset\mathbb{R}^n$ be a polytope and $f:\mathbb{R}^n\rightarrow\mathbb{R}$ be a $L$-smooth and $S$-strongly convex function. Consider the Away-Step Frank-Wolfe algorithm (AFW, Algorithm~\ref{afw}) with the line search strategy
 \begin{align*}
  \gamma_t\leftarrow
  \begin{cases}
   \argmin_{\gamma\in\left[0,1\right]}f(x_t+\gamma_t(v_t^\text{FW}-x_t))&\text{in the case of a FW step (Line~\ref{afw:fw})}\\
   \argmin_{\gamma\in\left[0,\gamma_{\max}\right]}f(x_t+\gamma_t(x_t-v_t^\text{away}))&\text{in the case of an away step (Line~\ref{afw:away})}.
  \end{cases}
 \end{align*}
 Then for all $t\in\llbracket0,T\rrbracket$,
 \begin{align*}
  f(x_t)-\min_\mathcal{P}f
  \leq\left(1-\frac{S}{4L}\left(\frac{W}{D}\right)^2\right)^{t/2}\left(f(x_0)-\min_\mathcal{P}f\right)
 \end{align*}
 where $D$ and $W$ are the diameter and the pyramidal width of $\mathcal{P}$ respectively.
\end{theorem}

We now show in Theorem~\ref{th:afw} that AFW can also achieve the convergence rate of Theorem~\ref{th:afwls} without line search, by using the short step strategy. We were later informed that this result was already derived by \citet{pedregosa20} in a more general setting.

\begin{theorem}[(AFW with short steps)]
\label{th:afw}
 Let $\mathcal{P}\subset\mathbb{R}^n$ be a polytope and $f:\mathbb{R}^n\rightarrow\mathbb{R}$ be a $L$-smooth and $S$-strongly convex function. Consider the Away-Step Frank-Wolfe algorithm (AFW, Algorithm~\ref{afw}) with the step-size strategy
 \begin{align}
 \gamma_t\leftarrow
 \begin{cases}
  \displaystyle\min\left\{\frac{\langle\nabla f(x_t),x_t-v_t^\text{FW}\rangle}{L\|x_t-v_t^\text{FW}\|_2^2},1\right\}&\text{in the case of a FW step (Line~\ref{afw:fw})}\\
  \displaystyle\min\left\{\frac{\langle\nabla f(x_t),v_t^\text{away}-x_t\rangle}{L\|v_t^\text{away}-x_t\|_2^2},\gamma_{\max}\right\}&\text{in the case of an away step (Line~\ref{afw:away})}.
 \end{cases}\label{afw:step}
\end{align}
Then for all $t\in\llbracket0,T\rrbracket$,
\begin{align*}
 f(x_t)-\min_\mathcal{P}f
 \leq\left(1-\frac{S}{4L}\left(\frac{W}{D}\right)^2\right)^{t/2}\left(f(x_0)-\min_\mathcal{P}f\right) 
\end{align*}
where $D$ and $W$ are the diameter and the pyramidal width of $\mathcal{P}$ respectively.
\end{theorem}

\begin{proof}
Let $t\in\llbracket0,T-1\rrbracket$ and denote $\epsilon_t\coloneqq f(x_t)-\min_\mathcal{P}f$. By geometric strong convexity (Lemma~\ref{lem:geom}),
\begin{align}
\epsilon_t\leq\frac{\langle\nabla f(x_t),v_t^\text{away}-v_t^\text{FW}\rangle^2}{2SW^2}.\label{afw:geom}
\end{align}
Furthermore, since
\begin{align*}
\langle\nabla f(x_t),v_t^\text{away}-v_t^\text{FW}\rangle 
=\langle\nabla f(x_t),v_t^\text{away}-x_t\rangle+\langle\nabla f(x_t),x_t-v_t^\text{FW}\rangle
\geq0
\end{align*}
we have
\begin{align}
 \max\{\langle\nabla f(x_t), v_t^\text{away}-x_t\rangle,\langle\nabla f(x_t),x_t-v_t^\text{FW}\rangle\}
 \geq\frac{\langle\nabla f(x_t), v_t^\text{away}-v_t^\text{FW}\rangle}{2}
 \geq0.\label{afw:max}
\end{align}
Note that AFW performs a step corresponding to $\max\{\langle\nabla f(x_t), v_t^\text{away}-x_t\rangle,\langle\nabla f(x_t),x_t-v_t^\text{FW}\rangle\}$ (Lines~\ref{afw:criterion} and~\ref{afw:criterion2}).
    
\emph{FW step.} In the case where the algorithm performs a FW step, we have $\gamma_t\leftarrow\min\{\langle\nabla f(x_t),x_t-v_t^\text{FW}\rangle/(L\|x_t-v_t^\text{FW}\|_2^2),1\}$, $x_{t+1}\leftarrow x_t+\gamma_t(v_t^\text{FW}-x_t)$, and by Line~\ref{afw:criterion} and~\eqref{afw:max},
\begin{align}
 \langle\nabla f(x_t), x_t -v_t^\text{FW}\rangle
 \geq\frac{\langle\nabla f(x_t), v_t^\text{away}-v_t^\text{FW}\rangle}{2}
 \geq0.\label{afw:fwgeq}
\end{align}
By smoothness of $f$,
\begin{align}
 \epsilon_{t+1}
 \leq\epsilon_t+\gamma_t\langle\nabla f(x_t),v_t^\text{FW}-x_t\rangle+\frac{L}{2}\gamma_t^2\|v_t^\text{FW}-x_t\|_2^2.\label{afw:fwsmooth}
\end{align}
Consider the choice of step-size~\eqref{afw:step} and suppose $\gamma_t<1$. Then, with~\eqref{afw:fwgeq} and~\eqref{afw:geom},
\begin{align*}
 \epsilon_{t+1}
 &\leq\epsilon_t-\frac{\langle\nabla f(x_t),x_t-v_t^\text{FW}\rangle^2}{2L\|x_t-v_t^\text{FW}\|_2^2}\\
 &\leq\epsilon_t-\frac{\langle\nabla f(x_t),v_t^\text{away}-v_t^\text{FW}\rangle^2}{8LD^2}\\
 &\leq\left(1-\frac{S}{4L}\left(\frac{W}{D}\right)^2\right)\epsilon_t.
\end{align*}
If $\gamma_t=1$ then $\langle\nabla f(x_t),x_t-v_t^\text{FW}\rangle\geq L\|x_t-v_t^\text{FW}\|_2^2$. By~\eqref{afw:fwsmooth}, the optimality of $v_t^\text{FW}$ (Line~\ref{afw:vfw}), and the convexity of $f$,
\begin{align*}
 \epsilon_{t+1}
 &\leq\epsilon_t+\langle\nabla f(x_t),v_t^\text{FW}-x_t\rangle+\frac{L}{2}\|v_t^\text{FW}-x_t\|_2^2\\
 &\leq\epsilon_t+\frac{\langle\nabla f(x_t),v_t^\text{FW}-x_t\rangle}{2}\\
 &\leq\epsilon_t+\frac{\langle\nabla f(x_t),x^*-x_t\rangle}{2}\\
 &\leq\frac{\epsilon_t}{2}.
\end{align*}
Therefore, the progress obtain by a FW step is
\begin{align}
 \epsilon_{t+1}
 &\leq\left(1-\min\left\{\frac{1}{2},\frac{S}{4L}\left(\frac{W}{D}\right)^2\right\}\right)\epsilon_t\nonumber\\
 &=\left(1-\frac{S}{4L}\left(\frac{W}{D}\right)^2\right)\epsilon_t\label{afw:fwprogress}
\end{align}
since $S\leq L$ and $W\leq D$.

\emph{Away step.} In the case where the algorithm performs an away step, we have $\gamma_t\leftarrow\min\{\langle\nabla f(x_t),v_t^\text{away}-x_t\rangle/(L\|v_t^\text{away}-x_t\|_2^2),\gamma_{\max}\}$, $x_{t+1}\leftarrow x_t+\gamma_t(x_t-v_t^\text{away})$, and by Line~\ref{afw:criterion} and~\eqref{afw:max},
\begin{align}
\langle\nabla f(x_t), v_t^\text{away}-x_t\rangle 
\geq\frac{\langle\nabla f(x_t), v_t^\text{away}-v_t^\text{FW}\rangle}{2}
\geq0.\label{afw:awaygeq}
\end{align}
By smoothness of $f$,
\begin{align}
 \epsilon_{t+1}
 \leq\epsilon_t+\gamma_t\langle\nabla f(x_t),x_t-v_t^\text{away}\rangle+\frac{L}{2}\gamma_t^2\|x_t-v_t^\text{away}\|_2^2\label{afw:awaysmooth}
\end{align}
Consider the choice of step-size~\eqref{afw:step} and suppose $\gamma_t<\gamma_{\max}$. Then, with~\eqref{afw:awaygeq} and~\eqref{afw:geom},
\begin{align}
 \epsilon_{t+1}
 &\leq\epsilon_t-\frac{\langle\nabla f(x_t),v_t^\text{away}-x_t\rangle^2}{2L\|v_t^\text{away}-x_t\|_2^2}\nonumber\\
 &\leq\epsilon_t-\frac{\langle\nabla f(x_t),v_t^\text{away}-v_t^\text{FW}\rangle^2}{8LD^2}\nonumber\\
 &\leq\left(1-\frac{S}{4L}\left(\frac{W}{D}\right)^2\right)\epsilon_t.\label{afw:awayprogress}
\end{align}
If $\gamma_t=\gamma_{\max}$ then $\gamma_{\max}\leq\gamma_t^*\coloneqq\langle\nabla f(x_t),v_t^\text{away}-x_t\rangle/(L\|v_t^\text{away}-x_t\|_2^2)$. Let $\varphi_t:\gamma\in\left[0,\gamma_t^*\right]\mapsto\gamma\langle\nabla f(x_t),x_t-v_t^\text{away}\rangle+L\gamma^2\|x_t-v_t^\text{away}\|_2^2/2$. By~\eqref{afw:awaysmooth},
\begin{align*}
 \epsilon_{t+1}
 \leq\epsilon_t+\varphi_t(\gamma_{\max}).
\end{align*}
The quadratic function $\varphi_t$ attains its unique global minimum at $\gamma=\gamma_t^*$ and satisfies $\varphi_t(\gamma_t^*)=-\langle\nabla f(x_t),x_t-v_t^\text{away}\rangle^2/(2L\|x_t-v_t^\text{away}\|_2^2)\leq0$ and $\varphi_t(0)=0$. Since $\gamma_{\max}\in\left[0,\gamma_t^*\right]$, we have $\varphi_t(\gamma_{\max})\leq0$ so 
\begin{align}
 \epsilon_{t+1}\leq\epsilon_t\label{afw:zero}
\end{align}
which shows that the progress is always nonnegative.

\emph{Wrapping up}. Since the steps with progress~\eqref{afw:zero} always decrease the number of vertices in the active set $\mathcal{S}_t$, there are at most $\floor{t/2}$ after $t$ iterations. By~\eqref{afw:fwprogress} and~\eqref{afw:awayprogress}, we conclude that
\begin{align*}
\epsilon_t
\leq\left(1-\frac{S}{4L}\left(\frac{W}{D}\right)^2\right)^{\ceil{t/2}}\epsilon_0
\leq\left(1-\frac{S}{4L}\left(\frac{W}{D}\right)^2\right)^{t/2}\epsilon_0.
\end{align*}
\end{proof}

\clearpage
\section{Proofs}
\label{apx:proofs}

\subsection{Preliminaries}

\begin{fact}[(Fact~\ref{fact:mu})]
Let $f:\mathcal{H}\rightarrow\mathbb{R}$ be $S$-strongly convex. Then $f$ is $S$-gradient dominated.
\end{fact}

\begin{proof}
 The function $f$ is strongly convex hence it has a unique minimizer, which we denote by $x^*\in\mathcal{H}$. Let $x\in\mathcal{H}$. By strong convexity, for all $y\in\mathcal{H}$ we have
 \begin{align*}
  f(y)\geq f(x)+\langle\nabla f(x),y-x\rangle+\frac{S}{2}\|y-x\|^2.
 \end{align*}
 Now we minimize both sides with respect to $y\in\mathcal{H}$. The left-hand side is minimized for $y=x^*$ and the right-hand side is minimized for $y=x-\nabla f(x)/S$. Thus,
 \begin{align*}
  f(x^*)\geq f(x)-\frac{\|\nabla f(x)\|^2}{2S}
 \end{align*}
 i.e.,
 \begin{align*}
  f(x)-\min_\mathcal{H}f\leq\frac{\|\nabla f(x)\|^2}{2S}.
 \end{align*}
\end{proof}

\subsection{Boosting via gradient pursuit}
\label{apx:boosting}

\begin{proposition}[(Proposition~\ref{prop})]
Let $t\in\llbracket0,T-1\rrbracket$ and suppose that $x_t\in\mathcal{C}$. Then:
 \begin{enumerate}[label=(\roman*)]
  \item\label{proof:prop1} $d_{1}$ is defined and $K_t\geq1$,
  \item\label{proof:prop2} $\lambda_0,\ldots,\lambda_{K_t-1}\geq0$,
  \item\label{proof:prop3} $d_k\in\operatorname{cone}(\mathcal{V}-x_t)$ for all $k\in\llbracket0,K_t\rrbracket$,
  \item\label{proof:prop4} $x_t+g_t\in\mathcal{C}$ and $x_{t+1}\in\mathcal{C}$,
  \item $\operatorname{align}(-\nabla f(x_t),g_t)\geq\operatorname{align}(-\nabla f(x_t),v_t-x_t)+(K_t-1)\delta$ where $v_t\in\argmin_{v\in\mathcal{V}}\langle\nabla f(x_t),v\rangle$ and $\operatorname{align}(-\nabla f(x_t),v_t-x_t)\geq0$.
 \end{enumerate}
 Since $x_0\in\mathcal{C}$, these properties are satisfied for all $t\in\llbracket0,T-1\rrbracket$ by induction.
\end{proposition}

\begin{proof}
We analyze BoostFW (Algorithm~\ref{boofw}).

 \begin{enumerate}[label=(\roman*)]
 \item We have $d_0=0$ so $r_0=-\nabla f(x_t)$ and $\operatorname{align}(r_0,d_0)=-1$ by definition~\eqref{angle}. Furthermore, since $v_0\in\argmax_{v\in\mathcal{V}}\langle r_0,v\rangle$ and $x_t\in\mathcal{C}$, we have 
 \begin{align*}
  \langle r_0,v_0-x_t\rangle
  =\langle r_0,v_0\rangle-\langle r_0,x_t\rangle
  \geq0
 \end{align*}
 so $u_0=v_0-x_t$ by Line~\ref{boofw:extra} since $d_0=0$. Thus, $\lambda_0\geq0$ (Line~\ref{boofw:lambda}) and $d_1'=\lambda_0(v_0-x_t)$ (Line~\ref{boofw:dprime}) so 
 \begin{align*}
  \operatorname{align}(-\nabla f(x_t),d_1')
  &=\frac{\langle-\nabla f(x_t),d_1'\rangle}{\|\nabla f(x_t)\|\|d_1'\|}\\
  &=\frac{\langle r_0,v_0-x_t\rangle}{\|r_0\|\|v_0-x_t\|}\\
  &\geq0\\
  &\geq-1+\delta\\
  &=\operatorname{align}(-\nabla f(x_t),d_0)+\delta.
 \end{align*}
 Therefore, by Line~\ref{criterion} the gradient pursuit procedure continues. 
 
  \item Let $k\in\llbracket0,K_t-1\rrbracket$. Since $v_k\in\argmax_{v\in\mathcal{V}}\langle r_k,v\rangle$ (Line~\ref{boofw:v}) and $x_t\in\operatorname{conv}(\mathcal{V})=\mathcal{C}$, 
  \begin{align*}
  \langle r_k,v_k-x_t\rangle
  =\max_{v\in\mathcal{V}}\langle r_k,v\rangle-\langle r_k,x_t\rangle
  \geq0.
  \end{align*}
  Thus, by Lines~\ref{boofw:extra}-\ref{boofw:lambda} we have
  \begin{align*}
   \lambda_k=\frac{\langle r_k,u_k\rangle}{\|u_k\|^2}\geq0.
  \end{align*}
  Furthermore, note that $-r_k$ is the gradient of the objective function in subproblem~\eqref{conicpb} and $\langle r_k,u_k\rangle$ is a scaled upper bound on its primal gap \citep{locatello17conic}. Thus, if $\lambda_k=0$ then the gradient pursuit procedure has already converged.
  
  \item We show by induction that $d_k\in\operatorname{cone}(\mathcal{V}-x_t)$ for all $k\in\llbracket0,K_t\rrbracket$. We have $d_0=0\in\operatorname{cone}(\mathcal{V}-x_t)$ so the base case is satisfied. Suppose that $d_k\in\operatorname{cone}(\mathcal{V}-x_t)$ for some $k\in\llbracket0,K_t-1\rrbracket$. If $u_k=v_k-x_t$ then $u_k\in\mathcal{V}-x_t$ and since $\lambda_k\geq0$ by~\ref{proof:prop2}, we have $d_{k+1}=d_k+\lambda_k(v_k-x_t)\in\operatorname{cone}(\mathcal{V}-x_t)$. Else, $u_k=-d_k/\|d_k\|$ so $d_{k+1}=(1-\lambda_k/\|d_k\|)d_k$ and it remains to show that $1-\lambda_k/\|d_k\|\geq0$. We will show that $1-\lambda_k/\|d_k\|\geq1/2$. We have
  \begin{align}
  1-\frac{\lambda_k}{\|d_k\|}\geq\frac{1}{2}
  &\Leftrightarrow\frac{1}{2}\geq\frac{\lambda_k}{\|d_k\|}=\frac{\langle r_k,-d_k/\|d_k\|\rangle}{\|d_k\|}\nonumber\\
   &\Leftrightarrow\frac{\|d_k\|^2}{2}\geq\langle r_k,-d_k\rangle.\label{prop:proof:0d}
  \end{align}
  so it suffices to show that $\|d_k\|^2/2\geq\langle r_k,-d_k\rangle$. Now, the procedure satisfies for all $k'\in\llbracket0,K_t-1\rrbracket$,
  \begin{align*}
   \|r_{k'+1}\|^2
   &=\|r_{k'}-\lambda_{k'}u_{k'}\|^2\\
   &=\|r_{k'}\|^2-2\lambda_{k'}\langle r_{k'},u_{k'}\rangle+\lambda_{k'}^2\|u_{k'}\|^2\\
   &=\|r_{k'}\|^2-\frac{\langle r_{k'},u_{k'}\rangle^2}{\|u_{k'}\|^2}\\
   &\leq\|r_{k'}\|^2
  \end{align*}
  where we used $\lambda_{k'}=\langle r_{k'},u_{k'}\rangle/\|u_{k'}\|^2$. Thus $\|r_k\|^2\leq\|r_0\|^2$, i.e., since $d_0=0$, $\|\nabla f(x_t)+d_k\|^2\leq\|\nabla f(x_t)\|^2$ so
  \begin{align*}
   \|\nabla f(x_t)\|^2
   \geq\|\nabla f(x_t)+d_k\|^2
   =\|\nabla f(x_t)\|^2+2\langle\nabla f(x_t),d_k\rangle+\|d_k\|^2
  \end{align*}
  hence
  \begin{align*}
   \langle\nabla f(x_t),d_k\rangle
   \leq-\frac{\|d_k\|^2}{2}.
  \end{align*}
  Thus,
  \begin{align*}
  \langle r_k,-d_k\rangle
  &=\langle\nabla f(x_t)+d_k,d_k\rangle\\
  &=\langle\nabla f(x_t),d_k\rangle+\|d_k\|^2\\
  &\leq\frac{\|d_k\|^2}{2}.
  \end{align*}
  Therefore, with~\eqref{prop:proof:0d} we can conclude that $d_{k+1}\in\operatorname{cone}(\mathcal{V}-x_t)$.
  
  \item By~\ref{proof:prop3}, $d_{K_t}\in\operatorname{cone}(\mathcal{V}-x_t)$ so since $g_t=d_{K_t}/\Lambda_t$, to show that $x_t+g_t\in\mathcal{C}$ it suffices to show that the sum of coefficients in the conical decomposition of $d_{K_t}$ is equal to $\Lambda_t$, and then it follows that $g_t\in\operatorname{conv}(\mathcal{V}-x_t)=\operatorname{conv}(\mathcal{V})-x_t=\mathcal{C}-x_t$. By Line~\ref{boofw:Lbd}, this is true and is verified by a simple induction on $k$: the base case is satisfied and if $u_k=v_k-x_t$ then $d_{k+1}=d_k+\lambda_k(v_k-x_t)$ and Line~\ref{boofw:Lbd} shows that $\Lambda_t\leftarrow\Lambda_t+\lambda_k$ is updated accordingly, else $u_k=-d_k/\|d_k\|$ so $d_{k+1}=(1-\lambda_k/\|d_k\|)d_k$ and Line~\ref{boofw:Lbd} shows that $\Lambda_t\leftarrow\Lambda_t(1-\lambda_k/\|d_k\|)$ is again updated accordingly. Thus, $x_t+g_t\in\mathcal{C}$. Then,
 \begin{align*}
 x_{t+1}
 &=x_t+\gamma_tg_t\\
 &=x_t+\gamma_t((x_t+g_t)-x_t)\\
 &=(1-\gamma_t)x_t+\gamma_t\underbrace{(x_t+g_t)}_{\in\mathcal{C}}.
 \end{align*}
 Since $\gamma_t\in\left[0,1\right]$, we conclude that $x_{t+1}\in\mathcal{C}$ by convex combination. 

 \item Since $d_0=0$, we have $r_0=-\nabla f(x_t)$ so by Line~\ref{boofw:v}, $v_0\in\argmin_{v\in\mathcal{V}}\langle\nabla f(x_t),v\rangle$, and we have $d_1=v_t-x_t$. Let $v_t\coloneqq v_0$. Since $g_t=d_{K_t}/\Lambda_t$ where $K_t\geq1$ by~\ref{proof:prop1}, by Line~\ref{criterion} we obtain
 \begin{align*}
  \operatorname{align}(-\nabla f(x_t),g_t)
  &=\operatorname{align}(-\nabla f(x_t),d_{K_t})\\
  &\geq\operatorname{align}(-\nabla f(x_t),v_t-x_t)+(K_t-1)\delta.
 \end{align*}
 Lastly,
 \begin{align*}
  \operatorname{align}(-\nabla f(x_t),v_t-x_t)
  &=\frac{\langle\nabla f(x_t),x_t-v_t\rangle}{\|\nabla f(x_t)\|\|x_t-v_t\|}\\
  &\geq0
 \end{align*}
 because $\langle\nabla f(x_t),x_t-v_t\rangle\geq0$ since $x_t\in\mathcal{C}$.
 \end{enumerate}
\end{proof}

\subsection{Convergence analysis}
\label{apx:cvproofs}

\begin{theorem}[(Theorem~\ref{th:boofwsc})]
 Let $f:\mathcal{H}\rightarrow\mathbb{R}$ be $L$-smooth, convex, and $\mu$-gradient dominated, and set $\gamma_t\leftarrow\min\{\eta_t\|\nabla f(x_t)\|/(L\|g_t\|),1\}$ or $\gamma_t\leftarrow\argmin_{\gamma\in\left[0,1\right]}f(x_t+\gamma g_t)$. Then for all $t\in\llbracket0,T\rrbracket$,
 \begin{align*}
  f(x_t)-\min_\mathcal{C}f
  \leq\frac{LD^2}{2}\prod_{s=0}^{t-1}\left(1-\eta_s^2\frac{\mu}{L}\right)^{\mathds{1}_{\{\gamma_s<1\}}}\left(1-\frac{\|g_s\|}{2\|v_s-x_s\|}\right)^{\mathds{1}_{\{\gamma_s=1\}}}
 \end{align*}
 where $v_s\in\argmin_{v\in\mathcal{V}}\langle\nabla f(x_s),v\rangle$ for all $s\in\llbracket0,T-1\rrbracket$.
\end{theorem}

\begin{proof}
 Let $\epsilon_t\coloneqq f(x_t)-\min_\mathcal{C}f$ for all $t\in\llbracket0,T\rrbracket$. We will prove the theorem for the step-size strategy $\gamma_t\leftarrow\min\{\eta_t\|\nabla f(x_t)\|/(L\|g_t\|),1\}$. The line search strategy follows since it achieves at least the same progress at every iteration. Let $t\in\llbracket0,T-1\rrbracket$. We have
 \begin{align}
  \eta_t
  &=\operatorname{align}(-\nabla f(x_t),g_t)\nonumber\\
  &=\frac{\langle-\nabla f(x_t),g_t\rangle}{\|\nabla f(x_t)\|\|g_t\|}.\label{proof:eta}
 \end{align}
 Suppose that $\gamma_t=\eta_t\|\nabla f(x_t)\|/(L\|g_t\|)$. Then since $f$ is $L$-smooth and is $\mu$-gradient dominated,
 \begin{align}
  \epsilon_{t+1}
  &\leq\epsilon_t+\gamma_t\langle\nabla f(x_t),g_t\rangle+\frac{L}{2}\gamma_t^2\|g_t\|^2\nonumber\\
  &=\epsilon_t-\gamma_t\eta_t\|\nabla f(x_t)\|\|g_t\|+\frac{L}{2}\gamma_t^2\|g_t\|^2\nonumber\\
  &=\epsilon_t-\eta_t^2\frac{\|\nabla f(x_t)\|^2}{2L}\nonumber\\
  &\leq\epsilon_t-\eta_t^2\frac{2\mu\epsilon_t}{2L}\nonumber\\
  &=\left(1-\eta_t^2\frac{\mu}{L}\right)\epsilon_t.\label{proof:g-1}
 \end{align}
 Else, $\eta_t\|\nabla f(x_t)\|/(L\|g_t\|)>1$ and $\gamma_t=1$. By~\eqref{proof:eta},
 \begin{align*}
  1
  &<\frac{\eta_t\|\nabla f(x_t)\|}{L\|g_t\|}\\
  &=\frac{\langle-\nabla f(x_t),g_t\rangle}{L\|g_t\|^2}
 \end{align*}
 so
 \begin{align*}
  L\|g_t\|^2<\langle-\nabla f(x_t),g_t\rangle.
 \end{align*}
 Hence, 
 \begin{align}
  \epsilon_{t+1}
  &\leq\epsilon_t+\gamma_t\langle\nabla f(x_t),g_t\rangle+\frac{L}{2}\gamma_t^2\|g_t\|^2\nonumber\\
  &=\epsilon_t+\langle\nabla f(x_t),g_t\rangle+\frac{L}{2}\|g_t\|^2\nonumber\\
  &<\epsilon_t+\frac{\langle\nabla f(x_t),g_t\rangle}{2}.\label{proof:gamma1}
 \end{align}
 Recall that $g_t=d_{K_t}/\Lambda_t$ and $d_1=\lambda_0(v_0-x_t)$ where $v_0\in\argmin_{v\in\mathcal{V}}\langle\nabla f(x_t),v\rangle$. If $K_t=1$ then $d_{K_t}=d_1$, else $\operatorname{align}(-\nabla f(x_t),d_{K_t})>\operatorname{align}(-\nabla f(x_t),d_1)$ by the condition in Line~\ref{criterion}. In both cases, we obtain
 \begin{align}
  \frac{\langle-\nabla f(x_t),g_t\rangle}{\|g_t\|}
  \geq\frac{\langle-\nabla f(x_t),v_0-x_t\rangle}{\|v_0-x_t\|}.\label{proof:etad}
 \end{align}
 Let $x^*\in\argmin_\mathcal{C}f$. By convexity and optimality of $v_0$,
 \begin{align}
  \epsilon_t
  &=f(x_t)-f(x^*)\nonumber\\
  &\leq\langle\nabla f(x_t),x_t-x^*\rangle\nonumber\\
  &\leq\langle\nabla f(x_t),x_t-v_0\rangle.\label{proof:v0}
 \end{align}
 Thus, with~\eqref{proof:gamma1} and~\eqref{proof:etad} we have
 \begin{align*}
  \epsilon_{t+1}<\left(1-\frac{\|g_t\|}{2\|v_0-x_t\|}\right)\epsilon_t.
 \end{align*}
 Therefore, together with~\eqref{proof:g-1} we conclude that for all $t\in\llbracket0,T\rrbracket$, 
 \begin{align*}
 \epsilon_t
 &\leq\epsilon_0\prod_{s=0}^{t-1}\left(1-\eta_s^2\frac{\mu}{L}\right)^{\mathds{1}_{\{\gamma_s<1\}}}\left(1-\frac{\|g_s\|}{2\|v_s-x_s\|}\right)^{\mathds{1}_{\{\gamma_s=1\}}}
 \end{align*}
 where $v_s\in\argmin_{v\in\mathcal{V}}\langle\nabla f(x_s),v\rangle$ for all $s\in\llbracket0,T-1\rrbracket$. We conclude by using the smoothness of $f$, the definition of $x_0$ (Line~\ref{boofw:x0}), and the convexity of $f$:
 \begin{align*}
  \epsilon_0
  &=f(x_0)-f(x^*)\\
  &\leq f(y)+\langle\nabla f(y),x_0-y\rangle+\frac{L}{2}\|x_0-y\|^2-f(x^*)\\
  &\leq f(y)+\langle\nabla f(y),x^*-y\rangle+\frac{LD^2}{2}-f(x^*)\\
  &\leq f(x^*)+\frac{LD^2}{2}-f(x^*)\\
  &=\frac{LD^2}{2}.
 \end{align*}
\end{proof}

\begin{theorem}[(Theorem~\ref{th:1/t})]
 Let $f:\mathcal{H}\rightarrow\mathbb{R}$ be $L$-smooth, convex, and $\mu$-gradient dominated, and set $\gamma_t\leftarrow\min\{\eta_t\|\nabla f(x_t)\|/(L\|g_t\|),1\}$ or $\gamma_t\leftarrow\argmin_{\gamma\in\left[0,1\right]}f(x_t+\gamma g_t)$. Consider Algorithm~\ref{boofw} with the minor adjustment $x_{t+1}\leftarrow x_t+\gamma_t'(v_t-x_t)$ in Line~\ref{update} when $\gamma_t=1$, where $v_t\leftarrow v_{k=0}$ is computed in Line~\ref{boofw:v} and $\gamma_t'\leftarrow\min\{\langle\nabla f(x_t),x_t-v_t\rangle/(L\|x_t-v_t\|^2),1\}$ or $\gamma_t'\leftarrow\argmin_{\gamma\in\left[0,1\right]}f(x_t+\gamma(v_t-x_t))$. Then for all $t\in\llbracket0,T\rrbracket$,
 \begin{align*}
  f(x_t)-\min_\mathcal{C}f
  \leq\frac{4LD^2}{t+2}.
 \end{align*}
\end{theorem}

\begin{proof}
 We consider the line search-free strategies; the line search strategies follow since they achieve at least the same progress at every iteration. Let $\epsilon_t\coloneqq f(x_t)-\min_\mathcal{C}f$ for all $t\in\llbracket0,T\rrbracket$. We will show by induction that for all $t\in\llbracket0,T\rrbracket$,
 \begin{align}
  \epsilon_t\leq\frac{4LD^2}{t+2}.\label{proof:1/t}
 \end{align}
 For the base case $t=0$, let $x^*\in\argmin_\mathcal{C}f$. By smoothness of $f$, the definition of $x_0$ (Line~\ref{boofw:x0}), and the convexity of $f$, we have
 \begin{align*}
  \epsilon_0
  &=f(x_0)-f(x^*)\\
  &\leq f(y)+\langle\nabla f(y),x_0-y\rangle+\frac{L}{2}\|x_0-y\|^2-f(x^*)\\
  &\leq f(y)+\langle\nabla f(y),x^*-y\rangle+\frac{LD^2}{2}-f(x^*)\\
  &\leq f(x^*)+\frac{LD^2}{2}-f(x^*)\\
  &=\frac{LD^2}{2}
 \end{align*}
 so \eqref{proof:1/t} holds. Suppose that~\eqref{proof:1/t} holds for some $t\in\llbracket0,T-1\rrbracket$.
 
 If $\gamma_t<1$ then we can proceed as in the proof of Theorem~\ref{th:boofwsc} and obtain
 \begin{align*}
  \epsilon_{t+1}
  &\leq\epsilon_t-\eta_t^2\frac{\|\nabla f(x_t)\|^2}{2L}.
 \end{align*}
 By Proposition~\ref{prop}\ref{prop:eta}, there exists $v_t\in\argmin_{v\in\mathcal{V}}\langle\nabla f(x_t),v\rangle$ such that
 \begin{align*}
  \eta_t
  &\geq\frac{\langle\nabla f(x_t),x_t-v_t\rangle}{\|\nabla f(x_t)\|\|x_t-v_t\|}.
 \end{align*}
 By convexity of $f$ and optimality of $v_t$,
 \begin{align}
  \epsilon_t
  &=f(x_t)-f(x^*)\nonumber\\
  &\leq\langle\nabla f(x_t),x_t-x^*\rangle\nonumber\\
  &\leq\langle\nabla f(x_t),x_t-v_t\rangle.\label{proof1/t:dual}
 \end{align}
 Thus,
 \begin{align*}
  \epsilon_{t+1}
  &\leq\epsilon_t-\frac{\langle\nabla f(x_t),x_t-v_t\rangle^2}{\|\nabla f(x_t)\|^2\|x_t-v_t\|^2}\frac{\|\nabla f(x_t)\|^2}{2L}\\
  &\leq\epsilon_t-\frac{\epsilon_t^2}{D^2}\frac{1}{2L}\\
  &=\epsilon_t\left(1-\frac{\epsilon_t}{2LD^2}\right).
 \end{align*}
 If $\epsilon_t\leq2LD^2/(t+2)$, then 
 \begin{align*}
  \epsilon_{t+1}
  &\leq\epsilon_t\\
  &\leq\frac{2LD^2}{t+2}\\
  &\leq\frac{4LD^2}{t+3}.
 \end{align*}
 Else, $2LD^2/(t+2)<\epsilon_t\leq4LD^2/(t+2)$ so
 \begin{align*}
  \epsilon_{t+1}
  &\leq\epsilon_t\left(1-\frac{\epsilon_t}{2LD^2}\right)\\
  &<\frac{4LD^2}{t+2}\left(1-\frac{1}{2LD^2}\frac{2LD^2}{t+2}\right)\\
  &=\frac{4LD^2}{t+2}\left(1-\frac{1}{t+2}\right)\\
  &\leq\frac{4LD^2}{t+3}
 \end{align*}
 so~\eqref{proof:1/t} holds for $t+1$.
 
 Now consider the case $\gamma_t=1$. Then by assumption, $x_{t+1}=x_t+\gamma_t'(v_t-x_t)$. By smoothness of $f$,
 \begin{align*}
  \epsilon_{t+1}
  \leq\epsilon_t+\gamma_t'\langle\nabla f(x_t),v_t-x_t\rangle+\frac{L}{2}\gamma_t'^2\|v_t-x_t\|^2.
 \end{align*}
 If $\gamma_t'=\langle\nabla f(x_t),x_t-v_t\rangle/(L\|x_t-v_t\|^2)$ then 
 \begin{align*}
  \epsilon_{t+1}
  &\leq\epsilon_t+\gamma_t'\langle\nabla f(x_t),v_t-x_t\rangle+\frac{L}{2}\gamma_t'^2\|v_t-x_t\|^2\\
  &=\epsilon_t-\frac{\langle\nabla f(x_t),x_t-v_t\rangle^2}{2L\|x_t-v_t\|^2}\\
  &\leq\epsilon_t\left(1-\frac{\epsilon_t}{2LD^2}\right)
 \end{align*}
 where we used~\eqref{proof1/t:dual}, and we can conclude as before.

 The final case to consider is $\gamma_t=1$ and $\gamma_t'=1$. Then $\langle\nabla f(x_t),x_t-v_t\rangle/(L\|x_t-v_t\|^2)\geq1$ so
 \begin{align*}
  \epsilon_{t+1}
  &\leq\epsilon_t+\gamma_t'\langle\nabla f(x_t),v_t-x_t\rangle+\frac{L}{2}\gamma_t'^2\|v_t-x_t\|^2\\
  &=\epsilon_t+\langle\nabla f(x_t),v_t-x_t\rangle+\frac{L}{2}\|v_t-x_t\|^2\\
  &\leq\epsilon_t+\frac{\langle\nabla f(x_t),v_t-x_t\rangle}{2}\\
  &\leq\frac{\epsilon_t}{2}\\
  &\leq\frac{2LD^2}{t+2}\\
  &\leq\frac{4LD^2}{t+3}
 \end{align*}
 where we used~\eqref{proof1/t:dual}. Therefore~\eqref{proof:1/t} holds for $t+1$.
\end{proof}

\begin{theorem}[(Theorem~\ref{th:boofwgood})]
 Let $f:\mathcal{H}\rightarrow\mathbb{R}$ be $L$-smooth, convex, and $\mu$-gradient dominated, and set $\gamma_t\leftarrow\min\{\eta_t\|\nabla f(x_t)\|/(L\|g_t\|),1\}$ or $\gamma_t\leftarrow\argmin_{\gamma\in\left[0,1\right]}f(x_t+\gamma g_t)$. Assume that $|\{s\in\llbracket0,t-1\rrbracket\mid\gamma_s<1,K_s>1\}|\geq\omega t^p$ for all $t\in\llbracket 0,T-1\rrbracket$, for some $\omega>0$ and $p\in\left]0,1\right]$. Then for all $t\in\llbracket0,T\rrbracket$,
 \begin{align*}
  f(x_t)-\min_\mathcal{C}f
  \leq\frac{LD^2}{2}\exp\left(-\delta^2\frac{\mu}{L}\omega t^p\right).
 \end{align*}
\end{theorem}

\begin{proof}
 Let $t\in\llbracket0,T\rrbracket$ and denote $N_t\coloneqq|\{s\in\llbracket0,t-1\rrbracket\mid\gamma_s<1,K_s>1\}|$. We have $N_t\geq\omega t^p$ and, by Proposition~\ref{prop}\ref{prop:eta}, if $K_t>1$ then $\eta_t\geq\delta$. By Theorem~\ref{th:boofwsc},
 \begin{align*}
  f(x_t)-\min_\mathcal{C}f
  &\leq\frac{LD^2}{2}\prod_{s=0}^{t-1}\left(1-\eta_s^2\frac{\mu}{L}\right)^{\mathds{1}_{\{\gamma_s<1\}}}\left(1-\frac{\|g_s\|}{2\|v_s-x_s\|}\right)^{\mathds{1}_{\{\gamma_s=1\}}}\\
  &\leq\frac{LD^2}{2}\prod_{\substack{s=0\\\gamma_s<1\\K_s>1}}^{t-1}\left(1-\eta_s^2\frac{\mu}{L}\right)\\
  &\leq\frac{LD^2}{2}\left(1-\delta^2\frac{\mu}{L}\right)^{N_t}\\
  &\leq\frac{LD^2}{2}\exp\left(-\delta^2\frac{\mu}{L}N_t\right)\\
  &\leq\frac{LD^2}{2}\exp\left(-\delta^2\frac{\mu}{L}\omega t^p\right).
 \end{align*}
\end{proof}

\begin{corollary}[(Corollary~\ref{cor:lmo})]
 In order to achieve $\epsilon$-convergence, the number of linear minimizations performed over $\mathcal{C}$ is 
 \begin{align*}
 \begin{cases}
  \displaystyle\mathcal{O}\left(\frac{LD^2\min\{K,1/\delta\}}{\epsilon}\right)&\text{in the worst-case scenario}\\
  \displaystyle\mathcal{O}\left(\min\left\{K,\frac{1}{\delta}\right\}\left(\frac{1}{\omega\delta^{2}}\frac{L}{\mu}\ln\left(\frac{1}{\epsilon}\right)\right)^{1/p}\right)&\text{in the practical scenario}.
  \end{cases}
 \end{align*}
\end{corollary}

\begin{proof}
 Let $t\in\llbracket0,T-1\rrbracket$. By Proposition~\ref{prop}\ref{prop:eta}, $\eta_t\geq(K_t-1)\delta$, and since $\eta_t\leq1$, we have $K_t\leq1/\delta+1$. Thus, $K_t\leq\min\{K,1/\delta+1\}$. We conclude by estimating $T$ via Theorem~\ref{th:1/t} or Theorem~\ref{th:boofwgood}: $T=\mathcal{O}(LD^2/\epsilon)$ or $T=\mathcal{O}\Big(\Big(\frac{1}{\omega\delta^2}\frac{L}{\mu}\ln\big(\frac{1}{\epsilon}\big)\Big)^{1/p}\Big)$ respectively.
\end{proof}

\subsection{Computational experiments}
\label{apx:expproofs}

\begin{fact}[(Fact~\ref{fact:2n})]
 Consider $\mathbb{R}^n$ and let $\tau>0$. Then $\mathcal{B}_1(\tau)=\{[z]_{1:n}-[z]_{n+1:2n}\mid z\in\tau\Delta_{2n}\}$.
\end{fact}

\begin{proof}
 Let $x\in\mathcal{B}_1(\tau)$. Define $\delta\coloneqq(\tau-\|x\|_1)/2n\geq0$ and $z\in\mathbb{R}^{2n}$ by 
 \begin{align*}
  [z]_i\coloneqq 
  \begin{cases}
   [x]_i+\delta&\text{if }[x]_i\geq0\\
   \delta&\text{if }[x]_i<0
  \end{cases}
  \quad\text{and}\quad
  [z]_{n+i}\coloneqq
  \begin{cases}
   \delta&\text{if }[x]_i\geq0\\
   -[x]_i+\delta&\text{if }[x]_i<0
  \end{cases}
 \end{align*}
 for all $i\in\llbracket1,n\rrbracket$. Then $x=[z]_{1:n}-[z]_{n+1:2n}$ and $z\geq0$.
 Furthermore, 
 \begin{align*}
  1^\top z
  &=\sum_{i=1}^n([z]_i+[z]_{n+i})\\
  &=\sum_{i=1}^n(|[x]_i|+2\delta)\\
  &=\|x\|_1+2n\delta\\
  &=\tau
 \end{align*}
 so $z\in\tau\Delta_{2n}$.
 
 For the reverse direction, let $z\in\tau\Delta_{2n}$ and $x\coloneqq[z]_{1:n}-[z]_{n+1:2n}$. Then
 \begin{align*}
  \|x\|_1
  &=\sum_{i=1}^n|[z]_i-[z]_{n+i}|\\
  &\leq\sum_{i=1}^n([z]_i+[z]_{n+i})\\
  &=\tau
 \end{align*}
 so $x\in\mathcal{B}_1(\tau)$.
\end{proof}

\end{document}